\documentclass[11pt,reqno,final]{amsart}
\usepackage{color}
\usepackage[colorlinks=true,allcolors=blue,backref=page]{hyperref}
\usepackage{amsmath, amssymb, amsthm}
\usepackage{mathrsfs}
\usepackage{mathtools}
\usepackage[noabbrev,capitalize,nameinlink]{cleveref}
\crefname{equation}{}{}
\usepackage{fullpage}
\usepackage[noadjust]{cite}
\usepackage{graphics}
\usepackage{pifont}
\usepackage{bbm}
\usepackage[T1]{fontenc}

\usepackage{environ}
\usepackage{framed}
\usepackage{url}
\usepackage[linesnumbered,ruled,vlined]{algorithm2e}
\usepackage[noend]{algpseudocode}
\usepackage[labelfont=bf]{caption}
\usepackage{cite}
\usepackage{framed}
\usepackage[framemethod=tikz]{mdframed}
\usepackage{appendix}
\usepackage{graphicx}
\usepackage[textsize=tiny]{todonotes}
\usepackage{tcolorbox}
\usepackage{enumerate}
\allowdisplaybreaks[4]
\usepackage{enumerate}
\usepackage{stmaryrd}

\usepackage[shortlabels]{enumitem}
\crefformat{enumi}{#2#1#3}
\crefrangeformat{enumi}{#3#1#4 to~#5#2#6}
\crefmultiformat{enumi}{#2#1#3}
{ and~#2#1#3}{, #2#1#3}{ and~#2#1#3}

\DeclareSymbolFont{symbolsC}{U}{pxsyc}{m}{n}
\SetSymbolFont{symbolsC}{bold}{U}{pxsyc}{bx}{n}
\DeclareFontSubstitution{U}{pxsyc}{m}{n}
\DeclareMathSymbol{\medcircle}{\mathbin}{symbolsC}{7}

\usepackage{thmtools}
\usepackage{thm-restate}

\crefname{algocf}{Algorithm}{Algorithms}

\crefname{equation}{}{} 
\AtBeginEnvironment{appendices}{\crefalias{section}{appendix}} 

\usepackage[color,final]{showkeys} 

\colorlet{refkey}{orange!20}
\colorlet{labelkey}{blue!30}

\crefname{algocf}{Algorithm}{Algorithms}

\numberwithin{equation}{section}
\newtheorem{theorem}{Theorem}[section]

\newtheorem{lemma}[theorem]{Lemma}

\crefname{claim}{Claim}{Claims}

\newtheorem*{question*}{Question}

\theoremstyle{definition}
\newtheorem{definition}[theorem]{Definition}

\newtheorem*{definition*}{Definition}

\theoremstyle{remark}
\newtheorem{remark}[theorem]{Remark}
\newtheorem*{remark*}{Remark}

\newcommand{\snorm}[1]{\lVert#1\rVert}
\newcommand{\sang}[1]{\langle #1 \rangle}

\newcommand{\mb}{\mathbb}

\newcommand{\mbm}{\mathbbm}
\newcommand{\mc}{\mathcal}
\newcommand{\mf}{\mathfrak}
\newcommand{\mr}{\mathrm}

\newcommand{\ol}{\overline}
\newcommand{\on}{\operatorname}

\newcommand{\wh}{\widehat}
\newcommand{\wt}{\widetilde}

\newcommand{\vect}[1]{\boldsymbol{#1}}
\newcommand{\eps}{\varepsilon}

\newcommand{\imod}[1]{~\mathrm{mod}~#1}

\newcommand\Q{\mathbf{Q}}

\newcommand\E{\mb{E}}
\renewcommand\P{\mb{P}}
\newcommand\GP{\operatorname{GP}}

\allowdisplaybreaks

\title{On polynomial progressions via transference}

\author[A1]{Daniel Altman}
\address{Department of Mathematics, Stanford University, 450 Jane Stanford Way,
Building 380, Stanford, CA 94305, USA}
\email{daniel.h.altman@gmail.com}

\author[A2]{Mehtaab Sawhney}
\address{Department of Mathematics, Columbia University, New York, NY 10027}
\email{m.sawhney@columbia.edu}

\begin{document}

\maketitle
\begin{abstract}
We prove new cases of reasonable bounds for the polynomial Szemer\'{e}di theorem both over $\mb{Z}/N\mb{Z}$ with $N$ prime and over the integers. In particular, we prove reasonable bounds for Szemer\'edi's theorem in the integers with fixed polynomial common difference. That is, we prove for any polynomial $P(y)\in \mb{Z}[y]$ with $P(0) = 0$, that the largest subset $A\subseteq [N]$ avoiding the pattern \[x, x+P(y),\ldots, x+ kP(y)\]
has size bounded by $\ll_{P,k}N(\log\log\log N)^{-\Omega_{P,k}(1)}.$
\end{abstract}

\section{Introduction}

Consider any set of distinct polynomials $P_1(y),\ldots,P_t(y)\in \mb{Z}[y]$ with $P_j(0) = 0$ for each $j$. Let $r_{P_1,\ldots,P_t}([N])$ denote the size of largest subset $A$ of $[N] := \{1,\ldots, N\}$ avoiding nontrivial copies of the pattern 
\[x + P_1(y),\ldots, x+P_t(y).\]
By nontrivial pattern here we require that $P_i(y) - P_j(y) \neq 0$ for all $i\neq j$. The celebrated polynomial Szemer\'{e}di theorem of Bergelson and Leibman \cite{BL96} states that 
\[r_{P_1,\ldots,P_t}([N]) = o_{P_1,\ldots,P_t}(N).\]

Due to the use of ergodic techniques the results of Bergelson and Leibman \cite{BL96} are qualitative and give no effective bounds for the rate of decay of $r_{P_1,\ldots,P_t}([N])/N$. In the case of Szemer\'{e}di's theorem \cite{Sze70,Sze75}, seminal work of Gowers \cite{Gow98,Gow01a} proved that 
\[r_{y,\ldots,ky}([N])\ll N(\log\log N)^{-\Omega_k(1)}.\]

However, progress towards genuinely polynomial cases of the polynomial Szemer\'{e}di's theorem remained more limited and was highlighted by Gowers \cite{Gow01} as a problem of interest. Results of Sark\"{o}zy \cite{Sar78} gave effective bounds for the pattern $x, x+y^2$ via the circle method. The underlying method has been generalized to give effective bounds for the case of $x,x+P(y)$ when $P(y)\in \mb{Z}[y]$ with $P(0) = 0$ by Lucier \cite{Luc06}.\footnote{This result applies more generally to the case where $P(y)$ is intersective. $P(y)$ is said to be intersective if for all integers $q\ge 1$ there is $n\in \mb{Z}$ such that $q|P(n)$.}

The cases where the underlying pattern contains more than two terms proved more difficult. A result of Green handled the case of three--term arithmetic progressions where the underlying difference was a sum of two squares \cite{Gre02}. A result of Prendiville \cite{Pre17} handled the case of $k$-term arithmetic progressions where the common difference is a perfect power. These results however both relied on homogeneity to pass to ``perfect power'' differences in order to run a density increment strategy. Therefore it was a substantial breakthrough when Peluse and Prendiville \cite{PP24,PP22} provided effective bounds for the case of $x,x+y,x+y^2$. This was extended substantially in work of Peluse \cite{Pel20} to the case of $x,x+P_1(y),\ldots,x+P_t(y)$ where the $P_j(y)$ are of differing degrees. Finally, work of Peluse, Sah and the second author \cite{PSS23} handled the case of $3$-term arithmetic progressions with difference $y^2-1$, and this was extended to general $P(y)$ with a simple integer root implicitly in work of Kravitz, Kuca and Leng \cite{KKL24}. 

The first result of this paper is a version of Szemer\'edi's theorem allowing for the difference to be equal to any fixed polynomial difference. 

\begin{theorem}\label{thm:main-int}
Fix a polynomial $P(y)\in \mb{Z}[y]$ with $P(0) = 0$. Let $\vec a = (a_1,\ldots, a_t) \in \mb{Z}^t$ with $a_i\neq a_j$ for $i\neq j$.  Then there exists $c = c(\on{deg}(P),t)>0$ such that the following hold.
\begin{itemize}
    \item If $t = 3$ and $P'(0) \neq 0$ then 
    \[r_{a_1P,\ldots,a_tP}([N])\ll_{P, \vec{a}} N\exp(-c(\log\log N)^{c}).\]
    \item If $P'(0)\neq 0$ then
    \[r_{a_1P,\ldots,a_tP}([N])\ll_{P, \vec{a}} N\exp(-c(\log\log\log N)^{c}).\]
    \item If $P'(0) = 0$ then 
    \[r_{a_1P,\ldots,a_tP}([N])\ll_{P, \vec{a}} N(\log\log\log N)^{-c}.\]
\end{itemize}
\end{theorem}
\begin{remark}
The final item in this theorem can likely be improved via improving a result of Prendiville \cite{Pre17} which handles the case $P(y) = y^k$ and invoking recent results regarding the inverse theorem for the Gowers norm \cite{LSS24b}. We refer the reader to \cref{rem:improve-bound}.
\end{remark}

We also consider the polynomial Szemer\'{e}di theorem over $\mb{Z}/N\mb{Z}$ where $N$ is prime. This model setting is often simpler than the corresponding setting over the integers as the image of a polynomial $Q(y)$ over $\mb{Z}/N\mb{Z}$ is dense. Given a sequence of polynomials $Q_1(y),\ldots,Q_t(y)$, we define $r_{Q_1,\ldots,Q_t}(\mb{Z}/N\mb{Z})$ to be the size of the largest subset of $\mb{Z}/N\mb{Z}$ avoiding copies of \[x + Q_1(y),\ldots, x+Q_t(y)\]
with $x\in \mb{Z}/N\mb{Z}$ and $y\in \mb{Z}/N\mb{Z}$ such that $Q_i(y) - Q_j(y)\neq 0$ for all $i\neq j$.

In $\mb{Z}/N\mb{Z}$, the first progress was due to work of Bourgain and Chang \cite{BC17}, giving quantitative bounds  for $r_{0,y,y^2}(\mb{Z}/N\mb{Z})$. This was generalized in work of Peluse \cite{Pel18} and independently Dong, Li and Sawin \cite{DLS20} to $r_{0,Q_1(y),Q_2(y)}(\mb{Z}/N\mb{Z})$ if $Q_1$ and $Q_2$ are linearly independent. These results were greatly extended in work of Peluse \cite{Pel19} to $x,x+Q_1(y),\ldots,x+Q_t(y)$ where $Q_i$ are linearly independent. Notably, this work introduced the \textit{degree lowering} method which was subsequently used over the integers. There are also additional results of Kuca \cite{Kuc21} which, for example, covers the case of $x,x+y,\ldots,x+(m-1)y,x+y^{m},\ldots,x+y^{m+k-1}$ and work of Leng \cite{Len24} which handles the case of $x,x+P(y),x+Q(y),x+P(y)+Q(y)$.

Our second result provides a unified framework which proves all known cases of  effective bounds in the polynomial Szemer\'{e}di theorem, as well as proving several new cases. Furthermore these results essentially resolve several conjectures of Leng \cite[Conjecture~1--3]{Len24} (although as stated there are counterexamples to \cite[Conjecture~1]{Len24}).

To state our results, we will require the notions of a \textit{kernel system} and \textit{homogeneous} polynomial patterns, and will introduce the notion of a \textit{transferable} polynomial pattern. The first two were considered in work of Kuca \cite{Kuc23} while the third arises when we apply transference. 

We first define the kernel system of a set of polynomials.
\begin{definition}\label{def:algebraic-complexity}
Given polynomials $P_1,\ldots,P_t\in\mb{Z}[x_1,\ldots,x_k]$, the \emph{kernel system} $\kappa(\mc{P})$ of pattern $\mc{P}=(P_1,\ldots,P_t)$ is the $\mb{Q}$-vector space of tuples of polynomials $\vec{Q}=(Q_1,\ldots,Q_t)\in\mb{Q}[z]$ such that
\[\sum_{i=1}^tQ_i(P_i(x_1,\ldots,x_k))=0.\]
\end{definition}

We now define the notions of homogeneous and transferable polynomial patterns. 
\begin{definition}\label{def:transferable-pattern}
Consider $\mc{P}=(x+P_1(y),\ldots,x+P_t(y))\in\mb{Z}[x,y]$ with $P_i(0)=0$ and $P_i$ distinct. Let $d=\max_{i\in[t]}\deg P_i$ and define the (unique) homogeneous linear polynomials $P_i^\ast\in\mb{Z}[y_1,\ldots,y_d]$ such that
\[P_i^\ast(y,y^2,\ldots,y^d)=P_i(y).\]
We say that $\mc{P}$ is \emph{transferable} if whenever $\vec{Q}=(Q_1,\ldots,Q_t)\in\kappa(\mc{P})$ holds we must in fact have
\[\sum_{i=1}^tQ_i(x+P_i^\ast(y_1,\ldots,y_d))=0.\]
We say that $\mc{P}$ is \emph{homogeneous} if $\kappa(\mc{P})$ is spanned by tuples of the form $(c_1z^i,\ldots,c_tz^i)$.
\end{definition}
\begin{remark}
By way of example, observe that the polynomial pattern $\mc{P}_1 = (x, x+y, x+2y, x+y^2)$ is not homogeneous due to the following identity: 
\[2(x+y^2) = 2x + x^2 - 2(x+y)^2 + (x+2y)^2.\]
We next have the existence of polynomial patterns which are not transferable, but which are homogeneous. In particular, consider $\mc{P}_2 = (x, x-y^2, x+y^2, x + y, x+y^3, x+y+y^3)$ and observe that 
\[(x+y^2)^2 - 2x^2 + (x-y^2)^2 = (x+y+y^3)^2 - (x+y^3)^2 - (x+y)^2 + x^2 = 2y^4.\]
Finally we have that all transferable patterns are homogeneous: the kernel system of a set of linear patterns is easily seen to be spanned by tuples of the form $(c_1z^i, \ldots, c_t z^i)$ by considering the degree of terms arising in a relation determined by $\vec{Q}\in \kappa(\mc{P})$. 
\end{remark}

The main result of this work over finite fields is the following transference result.
\begin{theorem}\label{thm:main-poly}
Let  $P_i(y)\in \mb{Z}[y]$ with $P_i(0) = 0$ and suppose that $\mc{P} = (x + P_1(y),\ldots,x+P_t(y))$ is transferable. Let $d = \max_{1\le i\le t}\on{deg}(P_i)$ and define $P_i^{\ast}(y_1,\ldots,y_d)$ as in \cref{def:transferable-pattern}. Let $N$ be prime and consider $f_i:\mb{Z}/N\mb{Z}\to \mb{C}$ with $\snorm{f_i}_{\infty}\le 1$. 

There exists $c = c_{\ref{thm:main-poly}}(P_1,\ldots,P_t)>0$ such that 
\[\Big|\mb{E}_{x,y\in \mb{Z}/N\mb{Z}}\prod_{i=1}^{t}f_i(x+P_i(y)) - \mb{E}_{x,y_1,\ldots,y_d\in \mb{Z}/N\mb{Z}}\prod_{i=1}^{t}f_i(x+P_i^{\ast}(y_1,\ldots,y_d))\Big|\ll_{\mc{P}} \exp(-c(\log \log N)^{c}).\]

Therefore there exists $c' = c_{\ref{thm:main-poly}}'(P_1,\ldots,P_t)>0$ such that if $\mc{P}$ is transferable then 
\[r_{\mc{P}}(\mb{Z}/N\mb{Z})\ll_{\mc{P}} N \exp(-c'(\log\log\log N)^{c'}).\]

\end{theorem}

\subsection{Proof outline}\label{ss:outline}
The proof of both \cref{thm:main-int} and \cref{thm:main-poly} rely on the use of transference to deduce our results regarding polynomial patterns from supersaturation results on other patterns. This general scheme of comparing counts and then applying supersaturation results stems from work of Wooley and Ziegler \cite{WZ12} regarding finding patterns $x+P_1(y),\ldots,x+P_t(y)$ where $y$ is additionally one less than a prime. This idea of transference has now been applied in the context of the polynomial Szemer\'{e}di theorem in several instances including work of Kuca \cite{Kuc21}, Leng \cite{Len24} and Peluse, Sah and the second author \cite{PSS23}. The main technical improvement in our work is avoiding the use of degree--lowering and instead noting that the required transference claim follows if it holds for nilsequences. We reduce to the more algebraic problem of nilsequence analysis via an iterative application of stashing and applications of the inverse theorem. The algebraic problem is then solved via adapting the arguments of the first author \cite{Alt22b} to transferable polynomial patterns. 

We first outline the proof of \cref{thm:main-poly} and then briefly discuss the technical differences which arise when handling \cref{thm:main-int}. For the sake of concreteness consider the case of $x,x+y,x+2y, x+y^3, x+2y^3$. Let $f_i:\mb{Z}/N\mb{Z}\to \mb{C},\  i=1,\ldots 5$ be one-bounded, that is, $\snorm{f_i}_\infty \leq 1$. We wish to upper bound the difference
\begin{align}
\Big|\mb{E}[f_1(x)f_2(x+y)&f_3(x+2y)f_4(x+y^3)f_5(x+2y^3)] \notag\\
&- \mb{E}[f_1(x)f_2(x+y)f_3(x+2y)f_4(x+z)f_5(x+2z)]\Big|\label{eq:transfer-intro}.
\end{align}
Given an appropriate upper bound, the claim about the polynomial Szemer\'edi theorem then follows from supersaturation results for variants of Szemer\'{e}di's theorem for linear patterns.

Suppose that 
\begin{align*}
\delta &\le \Big|\mb{E}[f_1(x)f_2(x+y)f_3(x+2y)f_4(x+y^3)f_5(x+2y^3)]\\
&\qquad\qquad\qquad\qquad- \mb{E}[f_1(x)f_2(x+y)f_3(x+2y)f_4(x+z)f_5(x+2z)]\Big|.
\end{align*}
We now perform a Cauchy--Schwarz maneuver known as stashing, due to Manners \cite{Man21}. Define the dual function
\begin{align*}
\mc{D}_1(x) = \mb{E}_{y,z}(f_2(x+y)f_3(x+2y)f_4(x+y^3)f_5(x+2y^3)-f_2(x+y)f_3(x+2y)f_4(x+z)f_5(x+2z)).
\end{align*}
By the Cauchy--Schwarz inequality, we have that 
\begin{align*}
\delta^2 \le \Big|\mb{E}[\mc{D}_1(x)f_2(x+y)f_3(x+2y)&f_4(x+y^3)f_5(x+2y^3)] \\
&- \mb{E}[\mc{D}_1(x)f_2(x+y)f_3(x+2y)f_4(x+z)f_5(x+2z)]\Big|.
\end{align*}
Applying a PET induction argument, we know that there exists $s \ge 1$ such that 
\begin{align*}
\Big|\mb{E}&[\mc{D}_1(x)f_2(x+y)f_3(x+2y)f_4(x+y^3)f_5(x+2y^3)] \\
&\qquad\qquad\qquad\qquad- \mb{E}[\mc{D}_1(x)f_2(x+y)f_3(x+2y)f_4(x+z)f_5(x+2z)]\Big|\ll \snorm{\mc{D}_1}_{U^s(\mb{F}_p)} + p^{-\Omega(1)}.
\end{align*}
Thus, ignoring the $p^{-\Omega(1)}$ term for the sake of simplicity in this outline, we find that $\snorm{\mc{D}_1}_{U^s(\mb{F}_p)}\gg \delta^{O(1)}$. By applying the inverse theorem for the Gowers norm \cite{LSS24c}, we obtain that there exists a nilsequence $\psi_1(x)$ (of appropriately bounded complexity) and $\delta_1 = \exp(-(\log(1/\delta))^{O(1)})$ such that 
\begin{align*}
\delta_1&\le \Big|\mb{E}[\psi_1(x)f_2(x+y)f_3(x+2y)f_4(x+y^3)f_5(x+2y^3)] \\
&\qquad\qquad\qquad\qquad- \mb{E}[\psi_1(x)f_2(x+y)f_3(x+2y)f_4(x+z)f_5(x+2z)]\Big|.
\end{align*}
We now iterate this procedure. In particular, we find nilsequences $\psi_2, \psi_3, \psi_4,\psi_5$ such that there is $\delta_2 = \exp(-(\log(1/\delta))^{O(1)})$ with 
\begin{align*}
\delta_2&\le \Big|\mb{E}[\psi_1(x)\psi_2(x+y)\psi_3(x+2y)\psi_4(x+y^3)\psi_5(x+2y^3)]\\ &\qquad\qquad\qquad\qquad- \mb{E}[\psi_1(x)\psi_2(x+y)\psi_3(x+2y)\psi_4(x+z)\psi_5(x+2z)]\Big|.
\end{align*}

This is now a largely algebraic problem of comparing the orbits of the polynomial sequences $g_{1}(x,y) = (g(x), g(x+y), g(x+2y), g(x+y^3), g(x+2y^3))$ and $g_{2}(x,y) = (g(x), g(x+y), g(x+2y), g(x+z), g(x+2z))$ on a nilmanifold of the form  $G^5/\Gamma^5$. Crucially, if $g:\mb{Z}\to G$ is suitably irrational, it transpires that $g_{1}(x,y)$ and $g_{2}(x,y)$ equidistribute on the same Leibman nilmanifold $G^{\Psi}/\Gamma^{\Psi} \leq G^5/\Gamma^5$. Such a result is proven without a quantitative treatment by Kuca \cite[Theorem~5.3]{Kuc23}. We provide a differing treatment of this result, utilizing the perspective in work of the first author \cite{Alt22b} in order to give quantitative bounds. We remark here that our analysis proceeds via a reduction to ``totally equidistributed'' nilsequences;  for further improvements to quantitative aspects it may be interesting to prove such a comparison via a ``step reduction'' procedure. 

The proof of \cref{thm:main-int} proceeds in a broadly similar manner; however certain technical differences arise in the course of the proof. First various algebraic considerations are simpler in this setting as we wish to compare the orbit $(g(x+a_1P(y)),\ldots,g(x+a_tP(y)))$ to that of $(g(x+a_1 y)),\ldots,g(x+a_t y))$. For our purposes, we may in fact fix $x$ and then compare orbits of the form $g^{\ast}(y)$ and $g^{\ast}(P(y))$, which substantially simplifies matters. 

This brief summary omits a number of technical details present over the integers which do not arise in the model setting of $\mb{Z}/N\mb{Z}$. We first introduce an archimedean weight on the associated measure for $y$ (as in \cite{PSS23}); this accounts for the fact that as $y$ grows the spacing of the polynomial $P(y)$ grows. Next when operating over the integers, $p$-adic biases of $P(y)$ arise, which we handle via the $W$-trick of Green \cite{Gre05}. This however is an oversimplification: observe that if $P'(0) = 0$ but $P''(0) \neq 0$ then comparing a $W$-trick version of $P(y)$ with $y$ is not prudent as they do not agree $p$-adically. In this case, one may hope to first apply the $W$-trick  to $P(y)$ and then compare the underlying progression to arithmetic progressions with square common difference. This nearly works; however, if $P(y) = y^2-y^4$ (say), observe that the common difference under consideration is almost always negative, and yet the transferred progression has positive common difference. This tension at ``infinity'' and $p$-adically requires us to find positivity from an additional pattern before comparing. In particular, instead of comparing differences coming from the $W$-tricked version of $P(y) = y^2-y^4$ (which is $y^2 - W^2y^4$) with $y^2$, we compare it with $y^2(-Wz+1)$. This requires bootstrapping Prendiville's result \cite{Pre17} to obtain positivity for arithmetic progressions with common difference $(-Wz + 1) y^2$ and then comparing to this modified pattern. 

\subsection{Organization of the paper and notation}
We prove  \cref{thm:main-int} in \cref{s:int} and prove  \cref{thm:main-poly} in  \cref{s:ff}. Various tools are developed within subsections; we refer to the beginning of each section for more refined outlines thereof. Throughout the paper we assume familiarity with various concepts associated with nilsequences. These conventions are developed in \cref{sec:nilmanifolds}. 

We use standard asymptotic notation for functions on $\mb{N}$ throughout. Given functions $f=f(x)$ and $g=g(x)$, we write $f=O(g)$, $f\ll g$, or $g\gg f$ to mean that there is a constant $C$ such that $|f(x)|\le Cg(x)$ for sufficiently large $x$. Subscripts indicate dependence on parameters. Additionally we let $[X] = \{1,\ldots,\lfloor X\rfloor\}$ and $[\pm X] = \{-\lfloor X\rfloor, \ldots, \lfloor X\rfloor\}$. 

\section{Transference in the integers}\label{s:int}
In this section we prove \cref{thm:main-int}. In the first subsection \cref{ss:hens}, we prove various $p$-adic comparison results between polynomials. These results are essentially extensions of classical Hensel lifting arguments. The second subsection \cref{ss:norms} introduces Gowers and Gowers--Peluse norms formally and deduces Gowers norm control for the relevant operators. While Gowers norm control follows immediately for the operator with differences $P(y)$ via the work of Peluse \cite{Pel20}, deducing the appropriate Gowers norm control for the operator with differences of the form $(\pm zW + 1)y^2$ requires an additional argument. In \cref{ss:nil-compar}, we establish the crucial nilseqeunce comparison result \cref{lem:nil-compar}. Finally in \cref{ss:stashing} we carry out the stashing procedure and deduce \cref{thm:main-int}. 

Before proceeding, we reduce to the case where the bottom coefficient of $P(y)$ is $1$. Suppose that $P(y) = \sum_{j=d'}^{d}b_jy^{j}$ with $b_j\in \mb{Z}$ where $d\ge d'\ge 1$. Firstly note that we may assume that $b_{d'}\geq 1$ by potentially replacing $(a_1,\ldots, a_t)$ with $-(a_1,\ldots, a_t)$. We claim that we may in fact assume that $b_{d'}=1$ without loss of generality (up to losing a constant factor implicit in \cref{thm:main-int}). To this end  observe that 
\[b_{d'}^{-d'-1}P(b_{d'}y) = \sum_{j=d'}^{d}b_jy^{j} \cdot b_{d'}^{j-d'-1}\in \mb{Z}[y].\]
Our claim then follows by subdividing $[N]$ into subprogressions with common difference $b_{d'}^{d'+1}$, rescaling, and replacing $P(y)$ by $b_{d'}^{-d'-1}P(b_{d'}y)$ and $N$ by $\lfloor b_{d'}^{-(d'+1)}N\rfloor$. Henceforth we assume that $b_{d'}=1$.

\subsection{Hensel lifting}\label{ss:hens}
A key feature of our results will be considering $W$-tricked versions of the polynomial $P(y)$. Let $w>1$ (we will later choose $w = (\log N)^{1/2}$, but it will be convenient to treat it as a parameter in the meantime). We set
\begin{equation}\label{eq:W}
W = \prod_{\substack{p\le w\\ (p, d) = 1}}p \cdot \prod_{\substack{p\le w\\(p,d!)\neq 1}}p^{2d} \cdot \prod_{\substack{p\le w^{1/2}}}p^{w^{1/3}},
\end{equation}
where here and throughout this section we will reserve $p$ to range over primes. Note that $\lfloor \log_p w \rfloor \leq w^{1/3}$, so we have importantly that $\on{lcm}(1,\ldots,w)|W$. 

Let 
\begin{equation}\label{eq:P_W}
    P_W(y) = W^{-d'}P(Wy)\in \mb{Z}[y].
\end{equation}
The point of this construction is that $P_W(y)$ will match the distribution of $y^{d'}$ modulo integers all of whose prime factors are at most $w$. We will now prove two lemmas towards this claim; they both follow from Hensel-type lifting.

\begin{lemma}\label{lem:linear-case}
Fix a prime $p$ and let $Q(y)= \sum_{j=1}^{d}c_jy^j\in \mb{Z}[y]$. Suppose that $p|c_j$ for $j\ge 2$ and $(c_1,p) = 1$. Then for all $k\ge 1$, we have for all $\ell\in [p^k]$ that 
\[\mb{P}_{y\in [p^k]}[Q(y) \equiv \ell \imod p^k] = p^{-k}.\]
\end{lemma}
\begin{proof}
We proceed by induction on $k$ and note that the base case of $k=1$ is immediate.

Suppose now that we have handled $k$. We would like to show that $Q(y')$ distributes uniformly modulo $p^{k+1}$ as $y'$ ranges over $[p^{k+1}]$. We  decompose $[p^{k+1}]$ as $p^k\cdot [p] +  [p^k]$, and let $y'=p^k y + b$, where $y\in [p]$ and $b\in [p^k]$.  
Observe that
\[Q(y') = Q(p^ky+b) = \left[Q(p^{k}y + b) - Q(b)\right] + Q(b)= \left[ \sum_{j=2}^{d}c_j((p^ky + b)^{j}-b^{j}) + c_1 p^{k}y\right] + Q(b).\]
Since $p| c_j$ for $j\geq 2$, we have that each term in the sum is equal to zero modulo $p^{k+1}$. Furthermore, by induction, $Q(b)$ distributes uniformly modulo $p^k$ as $b$ ranges over $[p^k]$. Finally, for any fixed multiple $a_b$ of $p^k$ which arises when instead viewing $Q(b)$ modulo $p^{k+1}$, we have that $c_1 y + a_b$ distributes uniformly over $\mb{Z}/p\mb{Z}$ as $y$ ranges over $[p]$, since $(c_1,p)=1$. This completes the proof.  
\end{proof}

\begin{lemma}\label{lem:mod-fixing}
Let $q\in \mb{N}$ be such that all prime factors of $q$ are at most $w$ and let $1\le r\le d$. Let $Q(y) = \sum_{j=r}^{d}c_jy^j\in \mb{Z}[y]$ satisfy the following properties for all $p \leq w$. Firstly, $(c_r, p) = 1$. Secondly,  $c_j \equiv 0\imod p$ for $j\ge r+1$, and if $(p,d!)\neq 1$, then $c_j \equiv 0\imod p^{2d}$ for $j\ge r+1$. Finally, let $\wt{c_r}\in \mb{Z}$ be such that $\wt{c_r} \equiv c_r \imod p$, and if $(p,d!) \neq 1$ then $\wt{c_r} \equiv c_r \imod p^{2d}$.

Then we have that 
\begin{equation}\label{eq:poly-hens}
\mb{P}_{y\in [q]}[Q(y) \equiv \ell \imod q] = \mb{P}_{y\in [q]}[\wt{c_r} y^{r} \equiv \ell \imod q].
\end{equation}
\end{lemma}
\begin{proof}
First observe via the Chinese remainder theorem that it suffices to prove the result for $q = p^{k}$ when $p\le w$ and $k\ge 1$. We now fix $p$ for the remainder of the analysis and proceed by induction on $k$. The case where $k = 1$ is immediate. 

We first consider the case where $(\ell,p)\neq 1$ and reduce this case by induction. Observe that if $p|Q(y)$ then $p|y$ and similarly if $p|\wt{c_r}y^r$ then $p|y$. Observe that if $p|y$ then $p^{r}$ divides $Q(y)$ and $\wt{c_r}y^r$. Replacing $\ell$ by $\ell \cdot p^{-r}$ and $Q(y)$ by $p^{-r}Q(py)$ we may reduce $k$ to $\max(k-r,0)$ and proceed by induction.

Therefore we may now assume that $(\ell,p) = 1$. We now break into cases based on whether or not $(p,d!)=1$. Assume firstly that $(p,d!) = 1$.  Then writing $[p^{k}] = p\cdot [p^{k-1}] + [p]$, we let $y'=py+b$ with $y\in [p^{k-1}]$ and $b\in [p]$. Computing similarly to the proof of \cref{lem:linear-case}, we write
\[Q(y') = [Q(py + b) - Q(b)] + Q(b) = \left[ \sum_{j=r+1}^{d}c_j((py+b)^j - b^{j}) + c_r  ((py+b)^{r} - b^{r}) \right] + Q(b).\]
Note that $p$ divides both $Q(py+b)-Q(b)$ and $\wt c_r (py+b)^r - \wt c_r b^r$, so since we have reduced to the case that $(\ell,p)=1$, it suffices to only consider $b$ with $(b,p)=1$. Fix now $b$ with $(b,p)=1$. With a view towards applying \cref{lem:linear-case} to $Q_1(y):=(Q(py+b)-Q(b))/p$, observe firstly that $p^2|\sum_{j=r+1}^{d}c_j((py+b)^j - b^{j})$. Thus the linear coefficient of $Q_1(y)$ is congruent to $rc_rb^{r-1}$ modulo $p$, which is nonzero modulo $p$ (since $(b,p)=1$, and since $r\leq d$ and $(p,d!)=1$). Furthermore observe that $p^2|((py+b)^{r} - b^{r} - r b^{r-1}  (py))$, so the quadratic and higher coefficients of $Q_1(y)$ are divisible by $p$. Thus we may apply \cref{lem:linear-case} to $Q_1(y)$, and a similar analysis shows that we may do the same to the polynomial $Q_2(y) := (\wt{c_r} (py+b)^{r} - \wt{c_r} b^r)/p$. This says that both $Q(py+b)-Q(b)$ and $\wt c_r (py+b)^r - \wt c_r b^r$ distribute uniformly over the multiples of $p$ in $\mb{Z}/p^{k}\mb{Z}$ as $y$ ranges over $[p^{k-1}]$. But note also that $Q(b) \equiv \wt c_r b^r \imod p$, and so we are able to conclude in this case (where, similarly to in the proof of \cref{lem:linear-case}, one uses that any multiple $a_b$ of $p$ obtained when viewing this modulo $p^{k}$ simply shifts the already uniform distribution of $Q(py +b)-Q(b)$ and $\wt c_r(py+b)^r-\wt c_r b^r$).

We now handle the case where $(p,d!)\neq 1$. In this case we have $c_j \equiv 0 \imod p^d$ for $j\geq r+1$, so we may assume without loss of generality that $d < k$. Then decompose $[p^{k}] = p^d \cdot [p^{k-d}] + [p^d]$, set $y'=p^{d} y +b$, and restrict to $(b,p)=1$ since the case $b\equiv 0 \imod p$ yields no solutions as before.  Then, 
\[Q(y') = [Q(p^dy + b) - Q(b)] + Q(b) = \left[ \sum_{j=r+1}^{d}c_j((p^dy+b)^j - b^{j}) + c_r  ((p^dy+b)^{r} - b^{r}) \right] + Q(b).\]

Observe that $p^{2d}|\sum_{j=r+1}^{d}c_j((p^{d}y+b)^j - b^{j})$, so the largest power of $p$ dividing the linear term of $Q(p^d y + b) - Q(b)$ is indeed $v_p(c_r r p^d b^{r-1}) = d + v_p(r)$ since $d + v_p(r) \le d + \lfloor \log_2(d)\rfloor < 2d$. Also note that $p^{2d}|c_r  ((p^{d}y+b)^{r} - b^{r} - r   b^{r-1}(p^{d}y))$. Thus taking $Q_1(y) = (Q(p^{d}y + b) - Q(b))/p^{d+ v_p(r)}$ and similarly $Q_2(y) = (\wt{c_r} (p^{d}y+b)^{r} - \wt{c_r}  b^r)/p^{d + v_p(r)}$, both have linear coefficient not divisible $p$, but quadratic and higher coefficients which are. We may then apply \cref{lem:linear-case} and deduce the required result much as we did in the previous case (where now we use that $Q(b) \equiv \wt c_r b^r \imod p^{d+\nu_p(r)}$).
\end{proof}

\subsection{Gowers norms and transferred operators}\label{ss:norms}
We first recall the standard Gowers norm.
\begin{definition}\label{def:gowers-norm}
Given $f\colon\mb{Z}/N\mb{Z}\to\mb{C}$ and $s\ge 1$, we define
\[\snorm{f}_{U^s(\mb{Z}/N\mb{Z})}^{2^s}=\mb{E}_{x,h_1,\ldots,h_s\in\mb{Z}/N\mb{Z}}\Delta_{h_1,\ldots,h_s}f(x)\]
where $\Delta_hf(x)=f(x)\ol{f(x+h)}$ is the multiplicative discrete derivative (extended to lists by composition). Given a natural number $N$ and a function $f\colon[N]\to\mb{C}$, we choose a number $\wt{N}\ge 2^{s}N$ and define $\wt{f}\colon\mb{Z}/\wt{N}\mb{Z}\to\mb{C}$ via $\wt{f}(x) = f(x)$ for $x\in[N]$ and $0$ otherwise. Then 
\[\snorm{f}_{U^s[N]} := \snorm{\wt{f}}_{U^s(\mb{Z}/\wt{N}\mb{Z})}/\snorm{\mbm{1}_{[N]}}_{U^s(\mb{Z}/\wt{N}\mb{Z})}.\]
\end{definition}
\begin{remark*}
This is known to be well-defined and independent of $\wt{N}$. Furthermore it is indeed a norm if $s\ge 2$; see \cite[Lemma~B.5]{GT10}.
\end{remark*}

We additionally recall the Gowers--Peluse norms which will be used when handling integer polynomial progressions.

\begin{definition}\label{def:box}
Given a function $f: \mb{Z} \to \mb{C}$ with finite support and $h,h'\in \mb{Z}$, we define 
\begin{equation}\label{double-difference} \Delta_{(h,h')}f(x) = f(x+h)\overline{f(x+h')}.\end{equation}
Given a multiset $\Omega = \{ \mu_1,\dots, \mu_k\}$ of probability measures on $\mb{Z}$ and a positive integer scale $N \ge 1$, we define the Gowers--Peluse norms as
\begin{equation}\label{gp-explicit-def} \snorm{f}_{U_{\on{GP}}[N; \Omega]}^{2^k} = \snorm{f}_{U_{\on{GP}}[N; \mu_1,\dots, \mu_k]}^{2^k} := \frac{1}{N}\sum_{x \in \mb{Z}} \mb{E}_{h_i,h_i' \sim \mu_i}\Delta_{(h_1,h_1')}\cdots\Delta_{(h_k,h_k')}f(x).\end{equation}

If $\Omega = \{\mu, \ldots,\mu\}$ where the measure $\mu$ is repeated $k$ times, we write that 
\[\snorm{f}_{U^k_{\on{GP}}[N; \mu]} = \snorm{f}_{U_{\on{GP}}[N; \Omega]}.\]
\end{definition}
When using the above definition, we will let $\mu_{E}$ denote the uniform measure on a multiset $E$ for the sake of simplicity. 

We now define the crucial counting operators which we will ultimately be used in our transference strategy. Recall that 
\[P(y) = \sum_{j=d'}^{d}b_jy^j\]
with $b_{d'} = 1$, and recall the definitions of $W$ and $P_W(y)$ from \eqref{eq:W} and \eqref{eq:P_W} respectively.

We will consider the following pair of operators. Given that these operators are rather more technical than the corresponding choices in (say) \cite{PSS23} we will provide motivation for the definitions after introducing them.
\begin{definition}\label{def:model}
Given functions $f_i:\mb{Z}\to \mb{C}$, $i=1, \ldots, t$,  we define the operator $\Lambda_W$ as follows. Let 
\[\Lambda_{W}(f_1,\ldots,f_t) = \mb{E}_{z\sim [1/2,1]}\mb{E}_{\substack{x\in [N]\\y\in [(zN/(|b_d|W^{d-d'}))^{1/d}]}}\prod_{i=1}^{t}f_i(x + a_i P_W(y)).\]
Let $\eps = b_{d}/|b_{d}|\in \{\pm 1\}$ and let 
\[\nu(y) = \Big(\frac{d'}{d} \cdot \Big(\frac{N}{W^2}\Big)^{(d-d')/(2dd')}\Big) \cdot y^{-(d-d')/d}.\]
We define the operator $\Lambda_{\on{Model}}$ by 
\[\Lambda_{\on{Model}}(f_1,\ldots,f_t) = \mb{E}_{\substack{x\in [N]\\y\in [(N/W^2)^{1/(2d')}]\\ z\in [N^{1/2}/2,N^{1/2}]}}\nu(y) \cdot \prod_{i=1}^{t}f_i(x + a_i(\eps\cdot zW + 1)y^{d'}).\]
\end{definition}

 To motivate these definitions, suppose for the sake of simplicity that we are attempting to transfer $3$-term arithmetic progressions with common difference $y^2 - y^4$, that is,  $\vec{a}=(0,1,2)$ and $P(y) = y^2-y^4$. A naive $W$-trick compares counts of the pattern $x, x+y^2 - W^2 y^4, x + 2(y^2 - W^2 y^4)$ with that of $x, x+y^2,x+2y^2$. As $y$ grows, the sequence $y^2 -W^2 y^4$ is far sparser than $y^2$; the weight $\nu(y)$ is the corresponding archimedean correction factor (a nearly identical device is used in \cite{PSS23}). Observe however that when $y$ is large, $y^2-W^2 y^4$ is negative whereas $y^2$ is positive. We correct this (without disturbing the $W$-trick) by comparing with $3$-APs with difference $(-zW + 1)y^2$. This however reintroduces an archimedian bias, which we then correct with the parameter $z$ in $\Lambda_{W}$. 

Our next task will be to prove that $\Lambda_{W}$ and $\Lambda_{\on{Model}}$ exhibit the appropriate Gowers norm control. We first obtain control on $\Lambda_{W}$; this is immediate given the results of Peluse \cite[Theorem~6.1]{Pel20}.
\begin{lemma}\label{lem:W-control}
There exist integers $s = s(d,t)\ge 1$ and $C = C(d,t)\ge 1$ such that the following holds. Let $\delta\in (0,1/2)$, $f_i:\mb{Z}\to \mb{C}$ be $1$-bounded functions supported on $[\pm N]$, $N\ge (\delta^{-1}W^{d})^{C}$ and suppose that 
\[|\Lambda_{W}(f_1,\ldots,f_t)|\ge \delta.\]
Then 
\[\min_{1\le j\le t}\snorm{f_j}_{U_{\on{GP}}^{s}[N;\mu_{W^{d-d'}\cdot [N/W^{(d-d')}]}]}\gg_{P,\vec{a}} \delta^{O_{d,t}(1)}.\]
\end{lemma}
Note that the implicit constant here only depends on the coefficients of $P(y)$ and not on the coefficients of $P_W(y)$.
\begin{proof}
Let $M_z = (zN/(|b_d|W^{d-d'}))^{1/d}$. Applying the triangle inequality we have that 
\[\mb{E}_{z\in [1/2,1]}\Big|\mb{E}_{\substack{x\in [N]\\y\in [M_{z}]}}\prod_{i=1}^{t}f_i(x + a_i P_W(y))\Big|\ge \delta.\]
Taking the supremum in $z$, there exist $z_0$ such that 
\[\Big|\mb{E}_{\substack{x\in [N]\\y\in [M_{z_0}]}}\prod_{i=1}^{t}f_i(x + a_i P_W(y))\Big|\ge \delta.\]
The desired result then follows from \cref{thm:int-control} which is essentially an immediate consequence of \cite[Theorem~6.1]{Pel20}. 
\end{proof}

We next obtain control for the operator given by $\Lambda_{\on{Model}}$. This proof requires opening up the inputs into concatenation estimates to a greater extent than the proof of \cref{lem:W-control}.

\begin{lemma}\label{lem:model-control}
There exist integers $s = s(d,t)\ge 1$ and $C = C(d,t)\ge 1$ such that the following holds. Let $\delta\in (0,1/2)$, $f_i:\mb{Z}\to \mb{C}$ be $1$-bounded functions supported on $[\pm N]$, $N\ge (\delta^{-1}W^{d})^{C}$ and suppose that 
\[|\Lambda_{\on{Model}}(f_1,\ldots,f_t)|\ge \delta.\]
Then 
\[\min_{1\le j\le t}\snorm{f_j}_{U_{\on{GP}}^{s}[N;\mu_{[N]}]}\gg_{d,\vec{a}} \delta^{O_{d,t}(1)}.\]
\end{lemma}
\begin{proof}
Let $\delta' = \delta^{O_{d}(1)}$ and taking the implicit constant suitably large, we have that
\[\Big| \mb{E}_{\substack{x\in [N]\\y\in [\delta' \cdot (N/W^2)^{1/(2d')}, (N/W^2)^{1/(2d')}]\\ z\in [N^{1/2}/2,N^{1/2}]}}\nu(y) \cdot \prod_{i=1}^{t}f_i(x + a_i(\eps\cdot zW + 1)y^{d'})\Big| \ge \delta/2.\]
Thus there exists $V \in [\delta'(N/W^2)^{1/(2d')},(N/W^2)^{1/(2d')}]$ such that
\[\Big| \mb{E}_{\substack{x\in [N]\\y\in [V,(1+\delta')V]\\ z\in [N^{1/2}/2,N^{1/2}]}}\nu(y) \cdot \prod_{i=1}^{t}f_i(x + a_i(\eps\cdot zW + 1)y^{d'})\Big| \ge \delta/2.\]
Note that $\nu(y)$ is nearly constant on the range $y\in[V,(1+\delta')V]$ and furthermore $\nu(y)\ll \delta'^{-1}$ on this range. This allows us to remove the weight $\nu(y)$ to obtain
\[\Big| \mb{E}_{\substack{x\in [N]\\y\in [V,(1+\delta')V]\\ z\in [N^{1/2}/2,N^{1/2}]}}\prod_{i=1}^{t}f_i(x + a_i(\eps\cdot zW + 1)y^{d'})\Big| \gg \delta^{O_{d}(1)}.\]
Now writing  $[(1+\delta')V] = [V]  \sqcup [V,(1+\delta')V]$, we have for at least one choice $V' = V$ or $V' = (1+\delta')V$ that 
\[\Big|\mb{E}_{z\in [N^{1/2}/2,N^{1/2}]}\mb{E}_{\substack{x\in [N]\\y\in [V']}}\prod_{i=1}^{t}f_i(x + a_i(\eps\cdot zW + 1)y^{d'})\Big|\gg \delta^{O_{d}(1)}.\]

Fixing $z$, we will apply concatenation results to the average over $x$ and $y$. To first obtain Gowers-Peluse norm control, we would like to apply \cref{thm:int-control}. However there is an unspecified dependence on the ratio $(|\eps \cdot zW+1| \cdot V'^{d'})/N$, which may depend on $\delta$. To remedy this, note that by an averaging argument (this time in the variable $x$), there exists a subinterval $I$ of $[N]$ of length $|I|=|\eps \cdot zW+1|V'^{d'} \gg \delta^{O_{d}(1)} N$ such that 
\[\Big|\mb{E}_{z\in [N^{1/2}/2,N^{1/2}]}\mb{E}_{\substack{x\in I\\y\in [V']}}\prod_{i=1}^{t}f_i(x + a_i(\eps\cdot zW + 1)y^{d'})\Big|\gg \delta^{O_{d}(1)}.\]
We now replace $f_i$ by $\wt{f_i}$, which is given by restricting $f_i$ to $I \pm 2\max_{i}|a_i|N^{1/2}W \cdot V'^{d'}$, so 
\[\Big|\mb{E}_{z\in [N^{1/2}/2,N^{1/2}]}\mb{E}_{\substack{x\in I\\y\in [V']}}\prod_{i=1}^{t}\wt{f_i}(x + a_i(\eps\cdot zW + 1)y^{d'})\Big|\gg \delta^{O_{d}(1)}.\]

We may now apply \cref{thm:int-control} and \cref{lem:box-norm-conv}(iii-iv) to obtain that
\[\mb{E}_{z\in [N^{1/2}/2,N^{1/2}]}\snorm{\wt{f_i}}^{2^s}_{U^s_{\on{GP}}[N;\mu_{(\eps \cdot zW + 1)\cdot [\pm N^{1/2}/W]}]}\gg_{\vec{a}} \delta^{O_{d,t}(1)}.\]
As stated $\wt{f_i}$ is a restriction of  $f_i$ to an interval; however one may Fourier expand the corresponding interval away  (via \cref{lem:major-arc-Fourier2}; see the proof of \cref{thm:int-control} for similar computations) to obtain that 
\[\mb{E}_{z\in [N^{1/2}/2,N^{1/2}]}\snorm{f_i}^{2^s}_{U^s_{\on{GP}}[N;\mu_{(\eps \cdot zW + 1)\cdot [\pm N^{1/2}/W]}]}\gg_{\vec{a}} \delta^{O_{d,t}(1)}.\]

We now use the key concatenation inequality in the work of Kravitz, Kuca, and Leng \cite[Corollary~6.4]{KKL24b}; we work with the statement in \cite[Lemma~5.5]{GS24}. Let $m$ be a power of $2$ to be chosen later.  Following on from the previous displayed equation, by \cite[Lemma~5.5]{GS24} there exists $\ell\ll_{m,s} 1$ such that 
\[\mb{E}_{\vec z \in [N^{1/2}/2,N^{1/2}]^\ell}\snorm{f_i}_{U_{\on{GP}}[N;\Omega[\vec{z}]]}\gg_{\vec{a}} \delta^{O_{d,t,m}(1)},\]
where 
\[\Omega'[\vec{z}] = (\mu_{(\eps\cdot z_{i_{k_1}}W + 1)\cdot [\pm N^{1/2}/W]} \ast \cdots \ast \mu_{(\eps\cdot z_{i_{k_m}}W + 1)\cdot [\pm N^{1/2}/W]})_{1\le k_1<k_2<\cdots<k_m\le \ell}\]
and $\Omega[\vec{z}]$ is obtained via duplicating the measures in $\Omega'[\vec{z}]$ $s$ times. 

The final claim for this proof is that for most choices of $z_{1},\ldots,z_m$, the measures given in $\Omega[\vec{z}]$ distribute mass over $[\pm N]$ roughly uniformly. In particular, we shall soon prove that if $m$ is larger than some absolute constant then 
\begin{equation}\label{eq:flat}
    \mb{E}_{z_1,\ldots,z_m\in [N^{1/2}/2,N^{1/2}]}\snorm{\mu_{(\eps\cdot z_{1}W + 1)\cdot [\pm N^{1/2}/W]} \ast \cdots \ast \mu_{(\eps\cdot z_{m}W + 1)\cdot [\pm N^{1/2}/W]}}_2^2\ll \frac{1}{N}.
\end{equation}
We will now complete the argument given this statement. Firstly note then by a union bound that
\[\mb{P}_{\vec{z}}[\max_{\nu\in \Omega[\vec{z}]}\snorm{\nu}_2^2\ge L/N]\ll_{m,s} \frac{1}{L}.\]
Thus taking $L = \delta^{-O_{d,t,m}(1)}$, we have that  
\[\mb{E}_{\vec{z}\in [N^{1/2}/2,N^{1/2}]^\ell}\mbm{1}[\max_{\nu\in \Omega[\vec{z}]}\snorm{\nu}_2^2\le L/N] \cdot \snorm{f_i}_{U_{\on{GP}}[N;\Omega[\vec{z}]]}\gg_{\vec{a}} \delta^{O_{d,t,m}(1)}.\]
The desired result then follows from \cite[Corollary~5.3]{GS24}. (This result provides control by a genuine Gowers norm; this may be converted to a Gowers--Peluse norm via \cref{lem:box-norm-conv} which gives that for $s\ge 2$ these norms are equivalent up to polynomial losses.)

It remains to show the key claim \eqref{eq:flat}; we do so via exponential sum estimates. Note that for $\Theta \in \mb{R}/\mb{Z}$,
\[|\hat{\mu}_{(\eps \cdot z W+1)\cdot [\pm N^{1/2}/W]}(\Theta)|^2 = \E_{h_1,h_2 \in [\pm N^{1/2}/W]} e(-\Theta(\eps \cdot z W+1)(h_1-h_2)),\] so by Parseval's identity 
\[\mb{E}_{z_1,\ldots,z_m\in [N^{1/2}/2,N^{1/2}]}\snorm{\mu_{(\eps\cdot z_{1}W + 1)\cdot [\pm N^{1/2}/W]} \ast \cdots \ast \mu_{(\eps\cdot z_{m}W + 1)\cdot [\pm N^{1/2}/W]}}_2^2 = \int_{\Theta}F(\Theta)^{m}~d\Theta,\]
where
\begin{align*}
F(\Theta) &= \mb{E}_{\substack{z\in [N^{1/2}/2,N^{1/2}]\\h_1,h_2\in [\pm N^{1/2}/W]}}e(\Theta (\eps \cdot z W + 1)(h_1-h_2))\\
&= \mb{E}_{\substack{z\in [1,N^{1/2}/2]\\h_1,h_2\in [2 N^{1/2}/W]}}e(\Theta (\eps \cdot (z+N^{1/2}/2) W + 1)(h_1-h_2)) \pm N^{-1/4}.
\end{align*}
Suppose that $|F(\Theta)|\ge \eta$. Then standard log--free exponential sum estimates (see e.g. \cite[Lemma~1.1.16]{Tao12}) imply that there exists $q\ll \eta^{-O(1)}$ such that 
\[\snorm{qW\Theta}_{\mb{R}/\mb{Z}}\ll \frac{\eta^{-O(1)}\cdot W}{N}\text{ and }\snorm{q\Theta}_{\mb{R}/\mb{Z}}\ll \frac{\eta^{-O(1)}\cdot W}{N^{1/2}}.\]
Since $W^{O(1)}\le N$ (by assumption), if $\eta^{-O(1)}\le N^{1/2}/W$ these inequalities together imply that there exists $q\ll \eta^{-O(1)}$ such that 
\[\snorm{q\Theta}_{\mb{R}/\mb{Z}}\ll \frac{\eta^{-O(1)}}{N}.\]
Upon choosing $m$ sufficiently large in terms of this $O(1)$ exponent of $\eta^{-1}$, the desired result follows via direct integration on the sizes of level sets of $|F(\Theta)|$; we omit the details. 
\end{proof}

\subsection{Nilsequence comparison result}\label{ss:nil-compar}
The main algebraic computation in this section is proving that one may compare the orbit of a nilsequence on $P_W(y)$ and $(\eps \cdot zW+1)y^{d'}$, and that the underlying averages will be close. This will be proven by iterative step--reduction, using the following equidistribution result of Leng \cite[Theorem~4]{Len23b}. We refer the reader to \cref{sec:nilmanifolds} for relevant definitions. 

\begin{theorem}\label{thm:step-equi}
Fix an integer $\ell\ge 1$,  and let $\delta\in(0,1/10)$, $M,D\ge 1$. Suppose that $G$ is a dimension $D$, at most $s$-step connected, simply connected, nilpotent Lie group with a given degree $k$ filtration, and the nilmanifold $G/\Gamma$ is complexity at most $M$ with respect to this filtration. Let  $F\colon G/\Gamma\to\mb{C}$ and let $g$ be an $\ell$-variable polynomial sequence on $G$ with respect to this filtration.

Furthermore suppose that $\snorm{F}_{\mr{Lip}}\le 1$ and $F$ has $G_{(s)}$-vertical frequency $\xi$ such that the height of $\xi$ is bounded by $M/\delta$. Suppose that $\min_{1\le i\le \ell} N_i\ge(M/\delta)^{O_{k,\ell}(D^{O_{k,\ell}(1)})}$ and
\[\big|\mb{E}_{\vec{n}\in[\vec{N}]}F(g(\vec{n})\Gamma)\big|\ge\delta.\]
There exists an integer $0\le r\le\dim(G/[G,G])$ such that:
\begin{itemize}
    \item We have horizontal characters $\eta_1,\ldots,\eta_r\colon G\to\mb{R}$ with heights bounded above by $(M/\delta)^{O_{k,\ell}(D^{O_{k,\ell}(1)})}$,
    \item For all $1\le i\le r$, we have $\snorm{\eta_i\circ g}_{C^{\infty}[\vec{N}]} \le(M/\delta)^{O_{k,\ell}(D^{O_{k,\ell}(1)})}$,
    \item For any $w_1,\ldots,w_s\in G/[G,G]$ such that $w_i$ lie in the joint kernel of $\eta_1,\ldots,\eta_r$, we have
    \[\xi([[[w_1, w_2], w_3],\ldots,w_s]) = 0.\]
\end{itemize}
\end{theorem}

Our main result in this subsection is a uniform bound on pairs of nilsequence averages. Handling this problem via an iterative reduction to the case of abelian nilsequences is a key trick in this work, whereas previous approaches have worked directly at the level of fixed step nilsequences. In what follows we will call a congruence class $\ell$ \textit{admissible} if there exists $y$ such that $P_W(y) \equiv \ell \imod W^{d-d'}$. We may also use this terminology for moduli other than $W^{d-d'}$. Recall that $P_W(y)$ has degree $d$ and leading coefficient $|b_d|W^{d-d'}$. In the statement and throughout its proof below we treat $d$ as fixed (that is, we allow all constants to depend on $d$). 

\begin{lemma}\label{lem:nil-compar}
Let $\delta\in (0,1/2)$ and $F(g(n)\Gamma)$ be a univariate nilsequence of degree $k$, dimension $D$ and complexity $M$. Furthermore suppose that $\max_{i} |b_i|\le M$, $N\ge W^{O_{k}(1)} \cdot (M/\delta)^{O_k(D^{O_{k}(1)})}$ and $w\ge (M/\delta)^{O_k(D^{O_{k}(1)})}$. 

Then for all $\ell \in [W^{d-d'}]$ which are admissible, we have that
\begin{align*}
\bigg|\mb{E}_{\substack{z\sim [1/2,1]\\y\in [(zN/(|b_d|W^{d-d'}))^{1/d}]\\P_W(y)\equiv \ell \imod W^{d-d'}}}F(g(P_W(y))\Gamma)- \mb{E}_{\substack{y\in [(N/W^2)^{1/(2d')}]\\ z\in [N^{1/2}/2,N^{1/2}]\\(\eps\cdot zW+ 1)y^{d'}\equiv \ell \imod W^{d-d'}}}\nu(y) F(g((\eps\cdot zW + 1)y^{d'})\Gamma)\bigg|\le \delta.
\end{align*}
\end{lemma}
\begin{remark}
A slightly subtle point (which initially caused the authors substantial confusion) is that such a comparison is best handled by considering residue conditions for the polynomials $P_W(y)$ and $(\eps\cdot zW+ 1)y^{d'}$, rather than for $y$ and $z$ separately. We note that this phenomenon does not arise when $P(y)$ has a simple root, in which case the the $W$-trick is rather simpler (cf. the difference between \cref{lem:linear-case} and \cref{lem:mod-fixing}).
\end{remark}

\begin{proof}
For the sake of simplicity, we will write $Q(z,y) = (\eps\cdot zW + 1)y^{d'}$ throughout. It will also be convenient to introduce the notation $V:= |b_d|W^{d-d'}$. Note that $V|W^{d-d'+1}$ since $|b_d| \leq w$. 

\textbf{Step 1: Reducing to the case of linear exponential phases.}
We will proceed in an inductive manner by first iteratively reducing the step of the underlying nilmanifold to where the nilmanifold is abelian. At each stage of the iteration $W_i$ will be a power of $W$. We will have that  $\ell_i$ is an admissible  residue class modulo $W_i$. 

At each stage, we insist that $g_i$ is a polynomial sequence on $G_i$, that $G_i/\Gamma_i$ has complexity $M_i$, and $F_i$ has Lipschitz norm at most $M_i$ and a vertical frequency $\xi_i$ of height at most $M_i$. Furthermore we assume that $G_i$ has step $s_i$, and that $s_i$ will decrease by $1$ at each stage. We may assume $\xi_i$ is nonzero or else $F_i$ descends to a function on the quotient $G_i/(G_i)_{(s_i)}$ and we may immediately induct downward. Furthermore, we maintain that the underlying group $G_i$ has dimension at most $D + i$. Finally, note that at each iteration we may replace $F_i(g_i(n)\Gamma_i)$ by updating the function to $F_i(\{g_i(0)\} \cdot \Gamma)$ and replacing the polynomial sequence by $\{g_i(0)\}^{-1}g(0)[g(0)]^{-1}$ in order to guarantee that $g_i(0) = \on{id}_{G_i}$ (one sees easily that doing so does not violate the shape of the complexity bounds). We will therefore assume that $g_i(0)=\on{id}_{G_i}$ without loss of generality. We will inductively maintain the condition that $\ell_i$ is a residue class modulo $W_i$ such that
\begin{align}\label{eq:step-induction}
\bigg|\mb{E}_{\substack{z\sim [1/2,1]\\y\in [(zN/V)^{1/d}]\\P_W(y)\equiv \ell_i \imod W_i}}F_i(g_i(P_W(y))\Gamma_i)- \mb{E}_{\substack{y\in [(N/W^2)^{1/(2d')}]\\ z\in [N^{1/2}/2,N^{1/2}]\\Q(z,y)\equiv \ell_i \imod W_i}}\nu(y) F_i(g_i(Q(z,y))\Gamma_i)\bigg|\ge \tau_i.
\end{align}
Our induction will evolve the parameters via:
\begin{align*}
 W_{i+1} &= W_i^{O_{k}(1)}\\
 M_{i+1} &\le (M_i/\tau_i)^{O_{k}(D^{O_{k}(1)})}\\
 \tau_{i+1} &\ge (M_i/\tau_i)^{-O_{k}(D^{O_{k}(1)})}.
\end{align*}

We now proceed with the induction, beginning with \eqref{eq:step-induction}. This implies that either 
\begin{align*}
\sup_{z\in [1/2,1]}\bigg|\mb{E}_{\substack{y\in [(zN/V)^{1/d}]\\P_W(y)\equiv \ell_i \imod W_i}}F_i(g_i(P_W(y))\Gamma_i)\bigg|\ge \tau_i/2
\end{align*}
or 
\begin{align*}
\bigg|\mb{E}_{\substack{y\in [(N/W)^{1/(2d')}]\\ z\in [N^{1/2}/2,N^{1/2}]\\Q(z,y)\equiv \ell_i \imod W_i}}\nu(y) F_i(g_i(Q(z,y))\Gamma_i)\bigg|\ge \tau_i/2.
\end{align*}
We will show that in either case we may conclude that for a suitable power $\wt W_i = W_i^{O_{k}(1)}$ and any $\wt y \in [\wt W_i]$,
\begin{equation}\label{eq:smooth-claim}
    \snorm{\eta_j ( g_i(\wt{W_i}y + \wt y))}_{C^{\infty}[N/\wt{W_i}]}\le (M_i/\tau_i)^{O(D^{O_{k}(1)})}.
\end{equation}
Here $\eta_j$ are a set of horizontal characters on $G_i$, each of height bounded by $(M_i/\tau_i)^{O_k(D^{O_{k}(1)})}$ such that $K := \cap_j \on{ker}(\eta_j)$ satisfies $\xi_i(K_{(s_i)}) = 0$

We begin with the first case. Write $g_i(n) = g_{i1}^{n}\cdots g_{ik}^{n^k}$ and let $y = W_i y' + \wt{y}$, where $\wt{y}$ varies over the values of $[W_i]$ that satisfy $P_W(\wt y) \equiv \ell_i \imod W_i$. Now fix a $\wt y$ for which the average over $y'$ is at least $\tau_i/2$; let $y_0$ be this value of $\wt y$. Then viewing $g_i(P_W(W_iy' + y_0))$ as a polynomial sequence in $y'$ which fails to equidistribute with respect to $F_i$, we obtain from \cref{thm:step-equi} a set of horizontal characters $\eta_j$, each of height bounded by $(M_i/\tau_i)^{O_k(D^{O_{k}(1)})}$, such that 
\begin{equation}\label{eq:smooth-PW}
    \snorm{\eta_j(g_i(P_W(W_i y' + y_0)))}_{C^\infty[(zN/V)^{1/d}/W_i]}\le (M_i/\tau_i)^{O_k(D^{O_{k}(1)})},
\end{equation}
and  $K := \cap_j \on{ker}(\eta_j)$ satisfies $\xi_i(K_{(s_i)}) = 0$. Now we may expand to view the above more explicitly as a polynomial in $y'$:
\begin{equation}\label{eq:eta-PW}
\eta_j(g_i(P_W(W_i y' + y_0)))  = \sum_{\ell=1}^k P_W(W_iy'+ y_0)^\ell \eta_j(g_{i\ell}) = \sum_{m=0}^{dk} \left(\sum_{\ell=\lfloor m/d\rfloor}^{k} c_{\ell, m} \eta_j(g_{i\ell}) \right) y'^{m}.
\end{equation}
Note in particular that after writing the above with respect to the basis $\binom{y'}{m}$, the coefficient of $\binom{y'}{dk}$ is $(dk)!\cdot V^k W_i^{dk}\cdot \eta_j(g_{ik})$. Absorbing the factor of $(dk)!$ into the definition of the character $\eta_j$, combining \eqref{eq:smooth-PW} and \eqref{eq:eta-PW} and of course recalling the definition of the smoothness norms \cref{def:smoothness} we obtain that
\[\snorm{V^k W_i^{dk}\cdot \eta_j(g_{ik})}_{\mb{R}/\mb{Z}}\le \frac{V^k W_i^{dk}\cdot (M_i/\tau_i)^{O_k(D^{O_{k}(1)})}}{N^k}.\]
Then, recalling that $V|W^{d-d'+1}$, 
 we may obtain a suitable power $\wt W_{i} = W_i^{O_{k}(1)}$ such that for any $\wt{y}\in [\wt W_i]$,
\begin{equation}\label{eq:pw-etagk-smooth}
    \snorm{P_W( \wt W_i y' + \wt y)^k \eta(g_{ik})}_{C^\infty[(N/V)^{1/d}/\wt W_i]} \leq (M_i/\tau_i)^{O_k(D^{O_k(1)})},
\end{equation}
where we use that $N\geq W^{O_{d,k}(1)} \cdot (M/\delta)^{O_k(D^{O_k(1)})}$ to deal with coefficients of $\binom{y'}{\ell}$ when $\ell < k$. Note also that \eqref{eq:smooth-PW} immediately implies
\[\snorm{\eta_j(g_i(P_W(\wt W_i y' + y_0)))}_{C^\infty[(N/V)^{1/d}/\wt W_i]}\le (M_i/\tau_i)^{O_k(D^{O_{k}(1)})}.\]
But then, letting $\wt g_i (n) := g_{i1}^n \cdots g_{i(k-1)}^{n^{k-1}}= g_i(n) g_{ik}^{-n^k}$, we may combine the previous displayed equation with \eqref{eq:pw-etagk-smooth} to deduce via the triangle inequality that 
\[\snorm{\eta_j(\wt g_i(P_W(\wt W_i y' + y_0)))}_{C^\infty[(N/V)^{1/d}/\wt W_i]}\le (M_i/\tau_i)^{O_k(D^{O_{k}(1)})}.\]
We have succeeded in obtaining a version of \eqref{eq:smooth-PW} but for a polynomial sequence $\wt g_i$ of lower degree. Thus we may apply this analysis iteratively downward on the degree of $g_i$. Doing so ultimately yields (at the cost of adjusting $\wt W_i$ but maintaining that $\wt W_i = W^{O_{k}(1)}$) for $\ell=1,\ldots,k$ that
\[\snorm{\wt W_i \cdot \eta_j(g_{il})}_{\mb{R}/\mb{Z}}\leq \frac{\wt W_i \cdot (M_i / \tau_i)^{O_k(D^{O_k(1)})}}{N^\ell}.\]
Writing out an expression like \eqref{eq:eta-PW} for $\eta_j(g_i(\wt W_iy + \wt y))$ allows us to deduce the claim \eqref{eq:smooth-claim} in the first case: 
\begin{equation}\label{eq:eta-g-smooth}
    \snorm{\eta_j ( g_i(\wt{W_i}y + \wt y))}_{C^{\infty}[N/\wt{W_i}]}\le (M_i/\tau_i)^{O_k(D^{O_{k}(1)})}.
\end{equation}

The argument for the second case is rather similar so we will only sketch it. Note that writing $[N^{1/2}]=[N^{1/2}/2] \sqcup [N^{1/2}/2,N^{1/2}]$, splitting the average into two intervals, and replacing $\tau_i/2$ by $\tau_i/8$, we may assume that the interval over which $z$ ranges starts at zero. The argument then proceeds in much the same way as the first case, where one does not need to worry about the leading coefficient $V$, but instead deals with a bivariate polynomial in $y,z$. Analysing the leading coefficient $z\binom{y}{k}$ allows for a similar downward induction on the degree of $g_i$, yielding similar information on the coefficients $g_{i\ell}$, and ultimately \eqref{eq:eta-g-smooth}. This concludes our casework; we now proceed from \eqref{eq:eta-g-smooth} and use it to interpret the averages in \eqref{eq:step-induction} on a lower step nilmanifold.

Let $W_{i+1}=W\cdot \wt W_i$; we will soon see that this is an appropriate modulus for the next step of the induction. To this end, we write the averages in \eqref{eq:step-induction} conditionally with respect to admissible values $\wt \ell_{i+1}$ of $\ell_i$ modulo $W_{i+1}$. By \cref{lem:mod-fixing}, 
\[ \mb{P}_{y\in [W_{i+1}]}(P_W(y) \equiv \wt \ell_{i+1} \imod W_{i+1}) = \mb{P}_{y,z\in [W_{i+1}]^2}(Q(z,y) \equiv \wt \ell_{i+1} \imod W_{i+1}),\]
so it follows by averaging that there exists a residue class $\ell_{i+1}$ modulo $W_{i+1}$ such that
\[\bigg|\mb{E}_{\substack{z\sim [1/2,1]\\y\in [(zN/V)^{1/d}]\\P_W(y)\equiv \ell_{i+1} \imod W_{i+1}}}F_i(g_i(P_W(y))\Gamma_i)- \mb{E}_{\substack{y\in [(N/W^2)^{1/(2d')}]\\ z\in [N^{1/2}/2,N^{1/2}]\\Q(z,y)\equiv \ell_{i+1} \imod W_{i+1}}}\nu(y) F_i(g_i(Q(z,y))\Gamma_i)\bigg|\ge 3\tau_i/4,\]
where we lose $\tau_i/4$, say, to account for the fact that the ranges of $y,z$ may not be exact multiples of $W_{i+1}$. Now let $\wt \ell_i \equiv \ell_{i+1}\imod \wt W_i$ (we will abusively also consider $\wt \ell_i$ to be an element of $[\wt W_i]$). Define $h(x) = g_i(\wt W_i x + \wt \ell_i)$.

Given \eqref{eq:eta-g-smooth} and via \cite[Lemma~A.1]{Len23b}, we may factor $h = \eps \cdot \wt{h} \cdot \gamma$ where $\gamma$ is $(M_i/\tau_i)^{O_k(D^{O_{k}(1)})}$-rational, $\eps$ is $((M_i/\tau_i)^{O_k(D^{O_{k}(1)})}, N/\wt W_i$)-smooth, and $\wt{h}$ takes values in $K=\cap_j \ker \eta_j$. The rationality of $\gamma$ implies that the sequence $\gamma(n)\Gamma_i$ is $(M_i/\tau_i)^{O_k(D^{O_{k}(1)})}$-periodic, and since $(M_i/\tau_i)^{O_k(D^{O_{k}(1)})} \leq w$, it is $W$-periodic. That is, $\gamma(n)\Gamma_i$ is constant on progressions modulo $W$. For brevity let $x_P=(P_W(y) - \wt \ell_i)/\wt W_i$ so $h(x_P)=g_i(P_W(y))$. Note that $x_P$ lies in a fixed progression modulo $W$ when $P_W(y)\equiv \ell_{i+1}\imod W_{i+1}$, so $\gamma(x_P)\Gamma_i=\gamma_0\Gamma_i$ is fixed; assume without loss of generality that $\gamma_0=\{\gamma_0\}$. Then updating $\gamma(n)$ to $\gamma_0^{-1} \gamma(n)$, $\wt{h}$ to $\gamma_0^{-1}\wt{h}\gamma_0$ and $\eps$ to $\eps \gamma_0$, we have reduced to the case that $g_i(P_W(y))\Gamma_i = \eps(x_P)\wt h(x_P)\Gamma_i$ and $g_i(Q(z,y))\Gamma_i=\eps(x_Q)\wt h(x_Q)\Gamma_i$, with the obvious notation.

To remove the dependence on the smooth part $\eps$, we now break $[N/\wt W_i]$ into segments of length $(N/\wt{W_i}) \cdot (M_i/\tau_i)^{-O_k(D^{O_{k}(1)})}$; call these segments $S_j = [N_j,N_{j+1}]$. Choosing the $(M_i/\tau_i)^{-O_k(D^{O_{k}(1)})}$ quantity appropriately, we may write $F_i(g_i(P_W(y))\Gamma_i)$ as $\sum_{j} \mbm{1}_{x_P\in S_j} F_i(\eps(N_j) \wt{h}(x_P) \Gamma_i)$ up to suitably small $L^{\infty}$ error ($\tau_i/8$, say). Doing the same for $Q(z,y)$ noting that the average of $\nu$ is bounded by 2, this implies that
\begin{align*}
\sum_{j}\bigg|&\mb{E}_{\substack{z\sim [1/2,1]\\y\in [(zN/V)^{1/d}]\\P_W(y)\equiv \ell_{i+1} \imod \wt{W_i}}}\mbm{1}_{x_P\in S_j} F_i(\eps(N_j) \wt{h}(x_P) \Gamma_i)\\
&-\qquad\qquad\mb{E}_{\substack{y\in [(N/W^2)^{1/(2d')}]\\z\in [N^{1/2}/2,N^{1/2}]\\Q(z,y)\equiv \ell_{i+1} \imod \wt{W_i}}}\nu(y) \cdot \mbm{1}_{x_Q\in S_j} F_i(\eps(N_j) \wt{h}(x_Q) \Gamma_i)\bigg|\ge \tau_i/4.
\end{align*}
Pigeonholing to find a large index $j$, we have that 
\begin{align*}
\bigg|&\mb{E}_{\substack{z\sim [1/2,1]\\y\in [(zN/V)^{1/d}]\\P_W(y)\equiv \ell_{i+1} \imod W_{i+1}}}\mbm{1}_{x_P\in S_j} F_i(\eps(N_j) \wt{h}(x_P) \Gamma_i)\\
&-\qquad\qquad\mb{E}_{\substack{y\in [(N/W^2)^{1/(2d')}]\\z\in [N^{1/2}/2,N^{1/2}]\\Q(z,y)\equiv \ell_{i+1} \imod W_{i+1}}}\nu(y) \cdot \mbm{1}_{x_Q\in S_j} F_i(\eps(N_j) \wt{h}(x_Q) \Gamma_i)\bigg|\ge (M_i/\tau_i)^{-O_k(D^{O_{k}(1)})}.
\end{align*}
We now can Fourier expand the indicator of $x_P\in S_j$ (via using \cref{lem:major-arc-Fourier2}), and after changing variables back to $P_W(y)$, obtain $\Theta\in \mb{R}/\mb{Z}$  such that 
\begin{align*}
\bigg|&\mb{E}_{\substack{z\sim [1/2,1]\\y\in [(zN/V)^{1/d}]\\P_W(y)\equiv \ell_{i+1} \imod W_{i+1}}}e(\Theta \cdot P_W(y))\cdot F_i(\eps(N_j) \wt{h}(x_P) \Gamma_i)\\
&-\qquad\qquad\mb{E}_{\substack{y\in [(N/W^2)^{1/(2d')}]\\z\in [N^{1/2}/2,N^{1/2}]\\Q(z,y)\equiv \ell_{i+1} \imod W_{i+1}}}\nu(y) \cdot e(\Theta \cdot Q(z,y))\cdot F_i(\eps(N_j) \wt{h}(x_Q) \Gamma_i)\bigg|\ge (M_i/\tau_i)^{-O_k(D^{O_{k}(1)})}.
\end{align*}

Finally, note that defining $h_{i+1}$ such that $h_{i+1}(P_W(y))= \wt h(x_P)$ defines a polynomial sequence in $K$. Take $G_{i+1} = \mb{R}\times (K/\ker \xi_i)$ and define $\Gamma_{i+1} = \mb{Z} \times ((\Gamma_i \cap K)/(\Gamma_i \cap \ker \xi_i))$. Furthermore we define $\wt{F_{i+1}}((x_1,x_2)\Gamma_{i+1}) = e(x_1) F_i(\eps(N_j) x_2\Gamma_i)$,  define $g_{i+1}(n) = (\Theta n, h_{i+1}(n))$ and set $\tau_{i+1}$ to be $(M_i/\tau_i)^{-O_k(D^{O_{k}(1)})}$. We obtain $F_{i+1}$ by taking a vertical Fourier expansion of $\wt{F_{i+1}}$. This completes the inductive step when $s\geq 2$.

Notice that the above inductive analysis did not require the step of the underlying nilmanifold to be larger than $1$ except (perhaps) in the final paragraph where our notation suggests that $\ker \xi_i \subseteq K$. In the case of a $1$-step nilmanifold (i.e. a torus), note that the iterative step reduces to the case of $\mb{R}/\mb{Z}$ and a function with a vertical frequency. Furthermore, running the argument in this case reduces to a linear phase function.

\textbf{Step 2: Handling the case of linear exponential phases.}

It therefore suffices to consider the case where $\wt{W} = W^{O(1)}$, $\ell\in \wt{W}$ is admissible and $\Theta\in \mb{R}/\mb{Z}$ such that
\begin{align*}
\bigg|\mb{E}_{\substack{z\sim [1/2,1]\\y\in [(zN/V)^d]\\P_W(y)\equiv \ell \imod \wt{W}}}e(\Theta\cdot  P_W(y))- \mb{E}_{\substack{y\in [(N/W^2)^{1/(2d')}]\\z\in [N^{1/2}/2,N^{1/2}]\\Q(z,y)\equiv \ell \imod \wt{W}}}\nu(y) e(\Theta \cdot Q(z,y))\bigg|\ge (M/\delta)^{-O_k(D^{O_{k}(1)})}.
\end{align*}

At least one of these averages must be suitably large, and after removing the weight $\nu(y)$ in the situation that only the latter average is large (e.g. as we did in \cref{lem:model-control}), an application of a log-free version of Weyl's inequality (see e.g. \cite[Lemma~1.1.16]{Tao12}) yields that $\Theta = \frac{a}{q} + \Theta'$ where $q\le W^{O(1)}$, $|\Theta'|\le W^{O(1)}/N$ and $(a,q) = 1$ (or $a = 0$ and $q=1$). We write $q = q_1q_2$ where $q_1$ has all primes factors of $q_1$ less than or equal to $w$ and $q_2$ has prime factors larger than $w$.

We first handle the case where $q_2\neq 1$. Note that each average above is invariant (up to suitably small error) under shifting the variables $y,z$ by any $t\in \wt W q_1 \cdot [q_2]$. Furthermore, since $|\Theta'|\le W^{O(1)}/N$, we have that $e(\Theta'\cdot  P_W(y+t))$ and $e(\Theta'\cdot P_W(y))$ differ negligibly when $t\in \wt{W}q_1 \cdot [q_2]$. Analogously, we have that $\nu(y+t_2) e(\Theta' \cdot Q(z+t_1,y+t_2))$ and $\nu(y) e(\Theta' \cdot Q(z,y))$ differ negligibly when $t_1,t_2\in \wt{W}q_1 \cdot [q_2]$. 

Therefore it suffices to bound
\begin{align*}
&\bigg|\mb{E}_{\substack{z\sim [1/2,1]\\y\in [(zN/V)^{1/d}]\\P_W(y)\equiv \ell \imod \wt{W}\\ t\in \wt{W}q_1 \cdot [q_2]}}e(\Theta'\cdot  P_W(y)) \cdot e(a/q \cdot P_W(y + t))\\
&\qquad\qquad\qquad\qquad- \mb{E}_{\substack{y\in [(N/W^2)^{1/(2d')}]\\z\in [N^{1/2}/2,N^{1/2}]\\Q(z,y)\equiv \ell \imod \wt{W}\\t_1,t_2\in \wt{W}q_1 \cdot [q_2]}}\nu(y) e(\Theta' \cdot Q(z,y))\cdot e(a/q \cdot Q(z+t_1,y+t_2))\bigg|
\end{align*}
Observe that we may write $a/q = a_1/q_1 + a_2/q_2$ and that the shifts by $t,t_1,t_2$ leave $y$ and $z$ invariant modulo $q_1$. Thus we have that the above is bounded by 
\begin{align*}
\le \Big|\mb{E}_{\substack{t\in[q_2]}}e(a_2/q_2 \cdot P_W(t))\Big|+ 2\Big|\mb{E}_{\substack{t_1,t_2\in[q_2]}}e(a_2/q_2 \cdot Q(t_1,t_2))\Big|
\end{align*}
where we use that the expectation of $\nu$ is bounded by $2$. These expectations are precisely exponential sums in $\mb{Z}/q_2\mb{Z}$ and we have that all the coefficients of $Q(t_1,t_2)$ and $P_W(t)$ are coprime to $q_2$. By using the Chinese remainder theorem, it suffices to consider the case where $q_2$ is a power of a prime $p$ larger than $w$. For $Q(t_1,t_2)$ the result is immediate as fixing any $t_2$ such that $(t_2,p) = 1$, we obtain a linear exponential sum with complete cancellation. In the case of $P_W(y)$, this is a complete exponential sum over a modulus $p^{b}$ with $p> w$ and $b \ge 1$. For $b >1$, it suffices to bound
\begin{align*}
\Big|\mb{E}_{\substack{t\in[p^{b}]}}e(a'/p^{b} \cdot P_W(t))\Big| &\le \mb{E}_{\substack{r\in[p^{b-1}]}}\Big|\mb{E}_{\substack{t\in[p]}}e(a'/p^{b} \cdot P_W(p^{b - 1} t + r))\Big| \\
&= \mb{E}_{\substack{r\in[p^{b-1}]}}\Big|\mb{E}_{\substack{t\in[p]}}e(a'/p^{b} \cdot p^{b - 1}t \cdot P_W'(r))\Big|;
\end{align*}
here $P_W'(y)$ denotes the derivative of $P_W(y)$ and $(a',p) = 1$. The inner sum evaluates to zero whenever $P_W'(r) \not\equiv  0 \imod p$ so noting that $P_W'$ is not identically zero modulo $p$ we obtain a bound of the form $\ll w^{-1}$ in this case as well. Finally the case of $b = 1$ follows via the Weil bound \cite[Theorem~11.23]{IK04} and we obtain a bound of the form $\ll w^{-1/2}$. 

Therefore we have reduced to the case where $q_2=1$. In this case we will in fact establish the stronger bound
\begin{align*}
&\bigg|\mb{E}_{\substack{z\sim [1/2,1]\\y\in [(zN/V)^{1/d}]\\P_W(y)\equiv \ell \imod \wt{W}}}e(\Theta\cdot  P_W(y)) - \mb{E}_{\substack{y\in [(N/W^2)^{1/(2d')}]\\z\in [N^{1/2}/2,N^{1/2}]\\Q(z,y)\equiv \ell \imod \wt{W}}}\nu(y) e(\Theta \cdot Q(z,y)))\bigg|\le N^{-\Omega_d(1)}.
\end{align*}
Observe that $\mbm{1}[P_W(y)\equiv \ell \imod \wt{W}] = \E_{t\in \wt{W}}e(t/\wt{W} \cdot (P_W(y) - \ell))$ and similarly for $Q(z,y)$. Via \cref{lem:mod-fixing}, and potentially replacing $\Theta$ by a fixed shift $\Theta + t_0/\wt{W}$, it suffices to prove that 
\begin{align*}
&\bigg|\mb{E}_{\substack{z\sim [1/2,1]\\y\in [(zN/V)^{1/d}]}}e(\Theta\cdot  P_W(y)) - \mb{E}_{\substack{y\in [(N/W^2)^{1/(2d')}]\\z\in [N^{1/2}/2,N^{1/2}]}}\nu(y) e(\Theta \cdot Q(z,y)))\bigg|\le N^{-\Omega_d(1)}.
\end{align*}
Recall that $\Theta = a/q + \Theta'$ with $|\Theta'|\le W^{O(1)}/N$ and $|q|\le W^{O(1)}$ with all prime factors at most $w$; note that this is unaffected by absorbing $t_0/\wt W$ into the $a/q$ part. It thus suffices to obtain the bound  
\begin{align*}
&\bigg|\mb{E}_{\substack{z\sim [1/2,1]\\y\in [(zN/V)^{1/d}]\\t\in [q]}}e(\Theta'\cdot  P_W(y))\cdot e(a/q\cdot P_W(y+t)) \\
&\qquad\qquad- \mb{E}_{\substack{y\in [(N/W^2)^{1/(2d')}]\\z\in [N^{1/2}/2,N^{1/2}]\\t_1,t_2\in [q]}}\nu(y) e(\Theta' \cdot Q(z,y)) \cdot e(a/q\cdot Q(z+t_1,y+t_2))\bigg|\le N^{-\Omega_d(1)}.
\end{align*}
Note that $(y+t)$ and $(z+t_1,y+t_2)$ range over $\mb{Z}/q\mb{Z}$ and $(\mb{Z}/q\mb{Z})^2$ respectively. 
As all prime factors of $q$ are at most $w$, we may apply \cref{lem:mod-fixing} and deduce that inner expectation over $t$ and $t_1,t_2$ is a constant independent of $y,z$ and is the same for both averages. Therefore it suffices to bound 
\begin{align*}
&\bigg|\mb{E}_{\substack{z\sim [1/2,1]\\y\in [(zN/(|b_d|W^{d-d'}))^{1/d}]}}e(\Theta'\cdot  P_W(y))- \mb{E}_{\substack{y\in [(N/W^2)^{1/(2d')}]\\z\in [N^{1/2}/2,N^{1/2}]}}\nu(y) e(\Theta' \cdot  (\eps \cdot zW + 1)y^{d'})\bigg|\le N^{-\Omega_d(1)}.
\end{align*}
Observing that $|y^{d-1}|\le N^{(d-1)/d}$ for the first integral and $|y^{d'}|\le N^{1/2}$ for the second, it suffices to prove that 
\begin{align*}
&\bigg|\mb{E}_{\substack{z\sim [1/2,1]\\y\in [(zN/(|b_d|W^{d-d'}))^{1/d}]}}e(\Theta'\cdot b_dW^{d-d'}y^{d})- \mb{E}_{\substack{y\in [(N/W^2)^{1/(2d')}]\\z\in [N^{1/2}/2,N^{1/2}]}}\nu(y) e(\Theta' \cdot \eps \cdot zWy^{d'})\bigg|\le N^{-\Omega_d(1)}.
\end{align*}

We now remove the contribution of $z$. Observe that if $z$ is a uniformly random integer in $[N^{1/2}/2,N^{1/2}]$, then $zN^{-1/2}$ is the uniform distribution on $[1/2,1]$ up to a $L_{\infty}$-error of $N^{-1/2}$, and so we have 
\[\bigg| \mb{E}_{z\in [N^{1/2}/2,N^{1/2}]}\nu(y)e(\Theta' \cdot \eps \cdot zWy^{d'}) - \mb{E}_{z\sim [1/2,1]} \nu(y) e(\Theta'\cdot \eps \cdot zW y^{d'} \cdot N^{1/2})\bigg| \ll W^{O(1)}N^{-1/2},\] uniformly for $y\in [(N/W^2)^{1/(2d')}]$. Therefore it suffices to obtain the bound 
\begin{align*}
&\sup_{z\in [1/2,1]}\bigg|\mb{E}_{\substack{y\in [(zN/(|b_d|W^{d-d'}))^{1/d}]}}e(\Theta'\cdot b_dW^{d-d'}y^{d})- \mb{E}_{\substack{y\in [(N/W^2)^{1/(2d')}]}}\nu(y) e(\Theta' \cdot \eps \cdot zWy^{d'}\cdot N^{1/2})\bigg|\le N^{-\Omega_d(1)}.
\end{align*}
Converting the Riemann sum to an integral, it suffices to obtain the same shape bound on 
\begin{align*}
&\sup_{z\in [1/2,1]}\bigg|((zN/(|b_d|W^{d-d'}))^{-1/d})\int_{0}^{(zN/(|b_d|W^{d-d'}))^{1/d}}e(\Theta'\cdot b_dW^{d-d'}y^{d})~dy\\
&\qquad\qquad\qquad-(N/W^2)^{-1/(2d')}\int_{0}^{(N/W^2)^{1/(2d')}}\nu(y)e(\Theta'N^{1/2}\cdot \eps \cdot zWy^{d'})~dy\bigg|.
\end{align*}
Via a change of variable this is equal to
\begin{align*}
&\sup_{z\in [1/2,1]}\bigg|(Nz)^{-1/d}\int_{0}^{(Nz)^{1/d}}e(\eps \cdot \Theta'\cdot y^{d})~dy-(Nz)^{-1/d'}\int_{0}^{(Nz)^{1/d'}}\nu(y\cdot (N^{1/2}\cdot zW)^{-1/d'})e(\Theta'\cdot \eps \cdot y^{d'})~dy\bigg|.
\end{align*}
This is equal (after substituting in for $\nu(y)$) to 
\begin{align*}
\sup_{z\in [1/2,1]}&\bigg|(Nz)^{-1/d}\int_{0}^{(Nz)^{1/d}}e(\eps \cdot \Theta'\cdot y^{d})~dy\\
&\qquad\qquad-(Nz)^{-1/d'}\cdot (d'/d)\int_{0}^{(Nz)^{1/d'}}y^{-(d-d')/d}(Nz)^{(d-d')/(dd')}e(\Theta'\cdot \eps \cdot y^{d'})~dy\bigg|.
\end{align*}
Observe that $(Nz)^{-1/d'} \cdot (Nz)^{(d-d')/(dd')} = (Nz)^{-1/d}$ and then making the change of variable $y' = y^{d'/d}$ the two integrals are (intentionally) identical. This completes the proof. 
\end{proof}

\subsection{Stashing and completing the proof}\label{ss:stashing}

We now establish the crucial positivity property of $\Lambda_{\on{Model}}(f,\ldots,f)$. In the following statement, implicit constants depend on $d, \vec{a}$.
\begin{lemma}\label{lem:positivity}
Consider a function $f:[N]\to [0,1]$ with $\mb{E}[f]\ge \delta$. We have the following bounds:
\begin{itemize}
    \item Suppose that $d' = 1$ and $t=3$. If $N\ge \exp(\log(1/\delta)^{O(1)})$, then 
    \[\Lambda_{\on{Model}}(f,\cdots,f)\ge \exp(-(\log(1/\delta))^{O(1)}).\]
    \item Suppose that $d' = 1$ and $t\ge 4$. If $N\ge \exp(\exp(\log(1/\delta)^{O(1)}))$, then 
    \[\Lambda_{\on{Model}}(f,\cdots,f)\ge \exp(-\exp(\log(1/\delta)^{O(1)})).\]
    \item Suppose that $d'>1$. If $N\ge \exp(\exp(1/\delta^{O(1)}))$, then 
    \[\Lambda_{\on{Model}}(f,\cdots,f)\ge \exp(-\exp(\delta^{-O(1)})).\]
\end{itemize}
\end{lemma}
\begin{proof}
In each case it suffices to prove for all $z\in [N^{1/2}/2,N^{1/2}]$ that 
\[\mb{E}_{\substack{x\in [N]\\y\in [(N/W^2)^{1/(2d')}]}} \prod_{i=1}^{t}f(x + a_i(\eps\cdot zW + 1)y^{d'})\]
is suitably large. We have used here that $\nu(y)\gg_d 1$. Observe that we may break $f$ into progressions modulo $|\eps \cdot zW + 1|$ and be left with arithmetic progressions of length $N^{1/2}/W$. At least a $\delta/3$-fraction of these have average at least $\delta/3$ and hence breaking into these progressions, shifting and rescaling, it suffices to establish a lower bound for the operator
\[\mb{E}_{\substack{x\in [N^{1/2}/W]\\y\in [(N/W^2)^{1/(2d')}]}}\prod_{i=1}^{t}\wt{f}(x +\eps \cdot a_iy^{d'})\]
where $\mb{E}_{x\in [N^{1/2}/W]}\wt{f}(x)\ge \delta/3$. When $d' = 1$, this is a standard supersaturation argument for (essentially) arithmetic progressions. The first bound follows from (an essentially trivial) modification of the results of Kelley--Meka \cite{KM23}. The second bound for $t=4$ follows from (trivial adaptions to) work of Green--Tao \cite{GT09} while the case of $t\ge 5$ follows from (trivial adaptions to)  work of Leng, Sah and the second author \cite{LSS24c}. The case $d'>1$ follows from noting that supersaturation applies to the work of Prendiville \cite{Pre17}. Let $\tau =\exp(-\exp(\delta^{-O(1)}))$. In particular, choosing any $m\in [\tau \cdot (N/W^2)^{1/(2d')}/2, \tau \cdot (N/W^2)^{1/(2d')}]$ we may break $x$ into progressions modulo $m^{d'}$ and consider $y$ which are multiples of $m$. Observe that this gives a change of variable which preserves the pattern (precisely as we are in the homogeneous case). An $\Omega(\delta)$-fraction of these have average at least $\Omega(\delta)$. Observe that these progressions have length $\tau^{-1}$ and hence any progression with density $\Omega(\delta)$ contains the desired pattern by \cite[Theorem~1.1]{Pre17}.
\end{proof}
\begin{remark}\label{rem:improve-bound}
It appears likely using that the configurations considered by Prendiville \cite{Pre17} are controlled by global Gowers norms via the results of Peluse \cite{Pel20}, and thus that the results of \cite{LSS24b} may be used to improve the obtained bounds. In particular, likely one may adopt the strategy in \cite{LSS24c} by at each stage passing to progressions where one guarantees that the common difference is an appropriate perfect power. We leave the details to the interested reader. 
\end{remark}

We now state the quasipolynomial inverse theorem of Leng, Sah and the second author \cite[Theorem~1.2]{LSS24b}. Various conventions regarding nilsequences are discussed carefully in \cref{sec:nilmanifolds}; our conventions are completely standard. 
\begin{theorem}\label{thm:main}
Fix $\delta\in (0,1/2)$. Suppose that $f\colon[N]\to\mb{C}$ is $1$-bounded and
\[\snorm{f}_{U^{s+1}[N]}\ge\delta.\]
Then there exists a nilmanifold $G/\Gamma$ of degree $s$, complexity at most $M$, and dimension at most $d$, a polynomial sequence $g(n)$ on $G$, as well as a function $F$ on $G/\Gamma$ which is at most $K$-Lipschitz such that 
\[|\mb{E}_{n\in[N]}[f(n)\ol{F(g(n)\Gamma)}]|\ge\eps,\]
where we may take
\[d\le\log(1/\delta)^{O_s(1)}\emph{ and }\eps^{-1},K,M\le\exp(\log(1/\delta)^{O_s(1)}).\]    
\end{theorem}

We now proceed to the proof of \cref{thm:main-int}.
\begin{proof}
Let $w = (\log N)^{1/2}$ and recall that 
\[W = \prod_{\substack{p\le w\\ (p, d) = 1}}p \cdot \prod_{\substack{p\le w\\(p,d!)\neq 1}}p^{2d} \cdot \prod_{\substack{p\le w^{1/2}}}p^{w^{1/3}}.\] Standard elementary estimates yield that $W\ll_{d}4^{w}$, which is subpolynomial in $N$.

We begin by supposing that $f_i:[N]\to \mb{C}$ are $1$-bounded and  
\[\Big|\Lambda_W(f_1,\ldots,f_t) - \Lambda_{\on{Model}}(f_1,\ldots,f_t)\Big|\ge \delta.\]
Define 
\[\Lambda_{\on{Diff}}(f_1,\ldots,f_t) := \Lambda_W(f_1,\ldots,f_t) - \Lambda_{\on{Model}}(f_1,\ldots,f_t),\]
and 
\begin{align*}
\mc{D}_{j}&(f_1,\ldots,f_{j-1},f_{j+1},\ldots,f_t)(x) \\
&:= \mbm{1}_{x\in [N]}\cdot \Big(\mb{E}_{z\sim [1/2,1]}\mb{E}_{y\in [(zN/(|b_d|W^{d-d'}))^{1/d}]}\prod_{i\in [t]\setminus \{j\}}f_i(x + (a_i-a_j) P_W(y))\\
&\qquad\qquad-\mb{E}_{\substack{y\in [(N/W^2)^{1/(2d')}]\\z\in [N^{1/2}/2,N^{1/2}]}}\nu(y) \cdot \prod_{i\in [t]\setminus \{j\}} f_i(x + (a_i-a_j)(\eps\cdot zW + 1)y^{d'})\Big).
\end{align*}
Observe that if $f_i$ are all supported on $[N]$ then 
\[\Lambda_{\on{Diff}}(f_1,\ldots,f_t) = \mb{E}_{x\in [N]}\mc{D}_{j}(f_1,\ldots,f_{j-1},f_{j+1},\ldots,f_t)(x) \cdot f_j(x).\] This implies via Cauchy--Schwarz that 
\[\mb{E}_{x\in [N]}[|\mc{D}_{j}(f_1,\ldots,f_{j-1},f_{j+1},\ldots,f_t)(x)|^2]\ge \delta^2.\]
However, observe that 
\begin{align*}
\mb{E}_{x\in [N]}&[|\mc{D}_{j}(f_1,\ldots,f_{j-1},f_{j+1},\ldots,f_t)(x)|^2]\\
&= \Lambda_{\on{Diff}}(f_1,\ldots,f_{j-1},\ol{\mc{D}_{j}(f_1,\ldots,f_{j-1},f_{j+1},\ldots,f_t)},f_{j+1},\ldots,f_t).
\end{align*}
Observe that combining \cref{lem:W-control} and \cref{lem:model-control}, we obtain that there exists $s=s(d,t)\ge 1$ such that 
\[\snorm{\mc{D}_{j}(f_1,\ldots,f_{j-1},f_{j+1},\ldots,f_t)}_{U_{\on{GP}}^s[N;\mu_{W^{d-d'}\cdot [N/W^{-(d-d')}]}]}\gg_{P,\vec{a}}\delta^{O_{d,t}(1)},\]
or
\[\snorm{\mc{D}_{j}(f_1,\ldots,f_{j-1},f_{j+1},\ldots,f_t)}_{U_{\on{GP}}^s[N;\mu_{[N]}]}\gg_{P,\vec{a}}\delta^{O_{d,t}(1)}.\]
Therefore we may apply the quasipolynomial inverse theorem for the Gowers norm \cite[Theorem~1.2]{LSS24b} combined with \cref{lem:box-norm-conv} to deduce that there exists $F_j$ which is $1$-bounded and is a nilsequence of complexity $\exp((\log(1/\delta)^{O(1)})$ along each progression of common difference $W^{d-d'}$ and that 
\[\Big|\mb{E}_{x\sim [N]}[F_j(x) \cdot \mbm{1}_{x\in [N]} \cdot \mc{D}_{j}(f_1,\ldots,f_{j-1},f_{j+1},\ldots,f_t)(x)]\Big|\gg \exp(-\log(1/\delta)^{O(1)}).\]
This is precisely equivalent to 
\[|\Lambda_{\on{Diff}}(f_1,\ldots, f_{j-1},F_j(\cdot) \cdot \mbm{1}_{\cdot\in [N]}\mbm,f_{j+1},\ldots,f_t)|\gg \exp(-\log(1/\delta)^{O(1)}).\]
This previous set of manipulations (in particular defining the associated dual function and then applying Cauchy--Schwarz in this way) has been termed stashing by Manners \cite{Man18}.

Iterating this procedure from $j=1$ to $j=t$, we may obtain 
\[|\Lambda_{\on{Diff}}(F_1(\cdot) \cdot \mbm{1}_{\cdot\in [N]},\ldots ,F_t(\cdot) \cdot \mbm{1}_{\cdot\in [N]})|\gg \exp(-\log(1/\delta)^{O(1)}).\]
Here $F_j$ are each nilsequences on progressions of common difference $W^{d-d'}$.

We now apply a Fourier expansion to $\mbm{1}_{\cdot\in [N]}$, and after possibly twisting $F_1$ by a phase and abusively denoting this as $F_1$, obtain that 
\[|\Lambda_{\on{Diff}}(F_1,\ldots ,F_t)|\gg \exp(-\log(1/\delta)^{O(1)}).\]

Expanding this condition out gives
\begin{align*}
\Big|\mb{E}_{z\sim [1/2,1]}&\mb{E}_{\substack{x\in [N]\\y\in [(zN/(|b_d|W^{d-d'}))^{1/d}]}}\prod_{i=1}^{t}F_i(x + a_i P_W(y))\\
&- \mb{E}_{\substack{x\in [N]\\y\in [(N/W^2)^{1/(2d')}]\\z\in [N^{1/2}/2,N^{1/2}]}}\nu(y) \cdot \prod_{i=1}^{t}F_i(x + a_i(\eps\cdot zW + 1)y^{d'})\Big|\ge \exp(-(\log(1/\delta)^{O(1)}).
\end{align*}
This implies that 
\begin{align*}
\sup_{x\in [N]}\Big|\mb{E}_{z\sim [1/2,1]}&\mb{E}_{\substack{y\in [(zN/(|b_d|W^{d-d'}))^{1/d}]}}\prod_{i=1}^{t}F_i(x + a_i P_W(y))\\
&- \mb{E}_{\substack{y\in [(N/W^2)^{1/(2d')}]\\z\in [N^{1/2}/2,N^{1/2}]}}\nu(y) \cdot \prod_{i=1}^{t}F_i(x + a_i(\eps\cdot zW + 1)y^{d'})\Big|\ge \exp(-\log(1/\delta)^{O(1)}).
\end{align*}
We next consider the set of residue classes attained by $P_W(y)$ modulo $W^{d-d'}$; call this set of residue classes $\mc{W}$ and consider $\ell \in \mc{W}$. Via \cref{lem:mod-fixing}, we have that 
\begin{align*}
\sup_{\substack{x\in [N]\\ \ell\in \mc{W}}}\Big|\mb{E}_{z\sim [1/2,1]}&\mb{E}_{\substack{y\in [(zN/(|b_d|W^{d-d'}))^{1/d}]\\P_W(y)\equiv \ell \imod W^{d-d'}}}\prod_{i=1}^{t}F_i(x + a_i P_W(y))\\
&- \mb{E}_{\substack{y\in [(N/W^2)^{1/(2d')}]\\z\in [N^{1/2}/2,N^{1/2}]\\(\eps\cdot zW + 1)y^{d'}\equiv \ell \imod W^{d-d'}}}\nu(y) \cdot \prod_{i=1}^{t}F_i(x + a_i(\eps\cdot zW + 1)y^{d'})\Big|\ge \exp(-\log(1/\delta)^{O(1)}).
\end{align*}
Fixing $x$ and $\ell \in \mc{W}$, we are exactly in the position to apply \cref{lem:nil-compar}. In particular, if $w\ge \exp(\log(1/\delta)^{O(1)})$ we obtain a contradiction. Therefore we have that 
\[\big|\Lambda_{\on{Diff}}(f_1,\ldots,f_t)\big|\le \exp(-(\log\log N)^{\Omega(1)}).\]

Now suppose that $A\subseteq [N]$ is free of nontrivial progressions of the form 
\[x+a_1P(y),\ldots,x+a_tP(y).\]
We then have that 
\[\big|\Lambda_{W}(\mbm{1}_A,\ldots,\mbm{1}_A)\big|\le N^{-\Omega(1)}.\]
The desired result then follows by combining our above bound on $\Lambda_{\on{Diff}}$ with \cref{lem:positivity}.
\end{proof}

\section{Transference in finite fields}\label{s:ff}

The main theorem for this section is the following result. By combining \cref{thm:transference} with standard supersaturation results we will immediately be able to deduce \cref{thm:main-poly}. We omit the details of this deduction because they are essentially the same as in the previous section. 
\begin{theorem}\label{thm:transference}
Let $\mc{P} = (x + P_1(y),\ldots,x+P_t(y))$ be a transferable polynomial pattern with $P_i(0) = 0$ for all $i$. Let $\mc{P}^{\ast}:=(x+P_1^\ast(\vect{y}),\ldots, x+P_t^\ast(\vect{y}))$ denote the associated transferred (linear) pattern. Let $f_i :\mb{Z}/N\mb{Z} \to \mb{C}$ be 1-bounded for $i=1,\ldots, t$ and $N$ be prime. Then 
\[\left|\Lambda_{\mc{P}}(f_1,\ldots,f_t) - \Lambda_{\mc{P}^\ast}(f_1,\ldots,f_t)\right| \leq \exp(-c_{\mc{P}}(\log \log N)^{c_{\mc P}}).\]
\end{theorem}

\subsection{Distributional tools}
We now develop the distributional tools on nilsequences which will play a crucial role in the proof of \cref{thm:transference}. In this section we will ultimately derive our results using totally equidistributed nilsequences. Therefore we essentially rely on the equidistribution results of Green and Tao \cite{GT12} (as quantified in \cite{TT21}); we first require a periodic variant. 

\begin{theorem}\label{thm:total-equi}
Fix an integer $\ell\ge 1$, $\delta\in(0,1/10)$, $M,D\ge 1$, and $F\colon G/\Gamma\to\mb{C}$. Suppose that $G$ is a dimension $D$ simply connected nilpotent Lie group with a given degree $k$ filtration, and the nilmanifold $G/\Gamma$ is complexity at most $M$ with respect to this filtration. Let $g$ be an $\ell$-variable polynomial sequence on $G$ with respect to this filtration.

Furthermore suppose that $\snorm{F}_{\mr{Lip}}\le 1$ and that $N\ge(M/\delta)^{\exp(D^{O_{k,\ell}(1)})}$ with $N$ prime. Furthermore suppose that $g$ is $N$-periodic in each coordinate (i.e. $g(\vec{n})g(\vec{n} + N \cdot e_{k})^{-1}\in \Gamma$ for all $\vec{n}$). 

If 
\[\Big|\mb{E}_{\vec{n}\in[(\mb{Z}/N \mb{Z})^{\ell}]}F(g(\vec{n})\Gamma) - \int_{G/\Gamma}F\Big|\ge\delta,\]
then we may write 
\[g(\vec{n}) = \eps \cdot g'(\vec{n}) \cdot \gamma(\vec{n})\]
where $g'$ takes values $G'$ with $\on{dim}(G')\le \on{dim}(G) - 1$, $\gamma(\vec{n})\in \Gamma$ and $\eps\in G$ such that $d_G(\eps,\on{id}_G)\le (M/\delta)^{\exp(D^{O_{k,\ell}(1)})}$. Furthermore if $g(0) = \on{id}_G$ then we may take $\eps = \on{id}_G$. 
\end{theorem}
\begin{proof}
Note that one may write $g(\vec{0}) = \{g(\vec{0})\}[g(\vec{0})]$ with $[g(\vec{0})]\in \Gamma$ and $\psi_{G}(\{g(\vec{0})\})\in [-1/2,1/2)^{D}$ (the Mal'cev coordinates of the second kind). Replacing $g$ by $\wt{g} = \{g(\vec{0})\}^{-1}g[g(\vec{0})]^{-1}$ and $F$ by $F(\{g(\vec{0})\} \cdot)$ we may assume $g(\vec{0}) = \on{id}_G$. 

Observe that 
\[\Big|\mb{E}_{\vec{n}\in(\mb{Z}/N\mb{Z})^{\ell}}F(g(\vec{n})\Gamma) - \int_{G/\Gamma}F\Big|\ge\delta\]
implies that for all $t\ge 1$
\[\Big|\mb{E}_{\vec{n}\in[[tN]^{\ell}]}F(g(\vec{n})\Gamma) - \int_{G/\Gamma}F\Big|\ge\delta.\]
By Fourier expansion (see e.g. \cite[Lemma~A.9]{PSS23}) and then applying \cref{thm:step-equi}, there exist a nonzero horizontal character $\eta_t$ such that $|\eta_t|\le (M/\delta)^{O_{k,\ell}\big(2^{D^{O_{k,\ell}(1)}}\big)}$ with 
\[\snorm{\eta_t \circ g}_{C^{\infty}[[tN]^{\ell}]}\le (M/\delta)^{O_{k,\ell}\big(2^{D^{O_{k,\ell}(1)}}\big)}.\]
By the Pigeonhole principle there exists an infinite subsequence of $t$ where $\eta_t = \eta$. Taking $t\to\infty$ along this infinite subsequence we obtain that  in fact
\[\snorm{\eta \circ g}_{C^{\infty}[[N]^{\ell}]}=0.\]

Then the proof of \cite[Lemma~A.1]{Len23b} yields 
\[g(\vec{n}) = g'(\vec{n}) \gamma(\vec{n})\]
where $\gamma$ is $Q$-periodic with $Q\le (M/\delta)^{\exp(D^{O_{k,\ell}(1)})}$, $g'(\vec{0}) = \gamma(\vec{0}) = 0$ and $g'\in \on{ker}(\eta)$. Letting $Q'$ be such that $QQ' \equiv 1 \imod N$, we have that
\begin{align*}
g(\vec{n}) &= g(QQ'\vec{n}) \cdot g(QQ'\vec{n})^{-1}g(\vec{n})\\
&=g'(QQ'\vec{n}) \cdot \gamma(Q(Q'\vec{n}))\cdot g(QQ'\vec{n})^{-1}g(\vec{n}).
\end{align*}
Observe that $\gamma(Q(Q'\vec{n}))\cdot g(QQ'\vec{n})^{-1}g(\vec{n})\in \Gamma$ as desired. To see that $g'(QQ'\vec{n})$ is appropriately periodic note that 
\begin{align*}
g'(QQ'(\vec{n}+Ne_{\ell})) &= g(QQ'(\vec{n}+Ne_{\ell}))\gamma^{-1}(QQ'(\vec{n}+Ne_{\ell}))\\
&=g(QQ'\vec{n})\cdot (g(QQ'\vec{n})^{-1}g(QQ'(\vec{n}+Ne_{\ell})))\cdot \gamma^{-1}(QQ'(\vec{n}+Ne_{\ell})) \\
&= g'(QQ'\vec{n}) \cdot \gamma(QQ'\vec{n})\cdot (g(QQ'\vec{n})^{-1}g(QQ'(\vec{n}+Ne_{\ell})))\cdot \gamma^{-1}(QQ'(\vec{n}+Ne_{\ell})).
\end{align*}
Noting that $\gamma(QQ'\vec{n})\in \Gamma$, $(g(QQ'\vec{n})^{-1}g(QQ'(\vec{n}+Ne_{\ell})))\in \Gamma$ and $\gamma^{-1}(QQ'(\vec{n}+Ne_{\ell}))\in \Gamma$ completes the proof. 
\end{proof}

We additionally require the following periodic factorization theorem.

\begin{theorem}\label{thm:factor}
Fix $A\ge 2$. Suppose that $G$ is a dimension $D$ simply connected nilpotent Lie group with a given degree $k$ filtration, and the nilmanifold $G/\Gamma$ is complexity at most $M'$ with respect to this filtration. Let $g$ be a polynomial sequence with respect to this filtration such that $g(0) = \on{id}_G$ and $g$ is $N$-periodic. 

Then there exists $M\in \bigg[M', (M')^{A^{(D+2)^{O_k(1)}}}\bigg]$ such that:
\begin{itemize}
    \item $g = g' \gamma$ with $g'(0) = \gamma(0) = \on{id}_G$, and $g'$ is $N$-periodic on $G'/(G'\cap \Gamma)$,
    \item $G'$ is an $M$-rational subgroup of $G$ and $G'/\Gamma'$ has a Mal'cev basis $\mc{X}'$ which comprises $M$-rational combinations of the Mal'cev basis $\mc{X}$ of $G/\Gamma$,
    \item Let $g' = g_1^{n}\cdots g_k^{n^k}$. For each nonzero $i$-th horizontal character $\eta$ of height at most $M^{A}$ on $G'$ (given filtration $G_i' = G'\cap G_i$) we have that 
    \[\snorm{\eta(g_i)}_{\mb{R}/\mb{Z}}\ge M^{A}\cdot N^{-i},\]
    \item $\gamma$ takes values only in $\Gamma$.
\end{itemize} 
\end{theorem}
\begin{proof}
A qualitative version of this proposition appears as \cite[Proposition 5.2]{CS14}. For the quantitative bound, \cite[Theorem~A.6]{TT21} tracks the quantitative details of the nonperiodic analogous factorization algorithm from \cite{GT12} and inputting \cref{thm:total-equi}, we get the stated quantitative bounds. 
\end{proof}

We now come to the crux of this section regarding the equidistribution of transferable polynomial patterns. We remark that we operate essentially throughout with the Lie algebra as in \cite{Alt22b}.

\begin{theorem}\label{thm:transfer}
Let $\delta\in (0,1/2)$, suppose that $G$ is a dimension $D$ simply connected nilpotent Lie group with a given degree $k$ filtration and the nilmanifold $G/\Gamma$ is complexity at most $M$ with respect to this filtration. Let $g$ be an $N$-periodic polynomial sequence with respect to this filtration with $g(0) = \on{id}_G$ and $N\ge (M/\delta)^{\exp(O_{\mc{P}}(D^{O_k(1)}))}$.

Let $\mc{P} = (x + P_1(y),\ldots,x+P_t(y))$ be a  transferable polynomial pattern with $P_i(0) = 0$. Let $\Psi:= \mc{P}^\ast$ denote the transferred pattern. Let $G^\Psi/\Gamma^\Psi$ denote the corresponding Leibman nilmanifold. 

Consider an $M$-Lipschitz function $F:G^{\Psi}/\Gamma^{\Psi}\to \mb{C}$ such that 
\begin{equation}\label{eq:noneq}
\bigg|\mb{E}_{x,y\in \mb{Z}/N\mb{Z}}F((g(x+ P_1(y)),\ldots,g(x+ P_t(y)))) - \int_{G^{\Psi}/\Gamma^{\Psi}} F \bigg|\ge \delta.
\end{equation}

Let the Taylor expansion of $g(n)$ be 
\[g(n) = \prod_{j=1}^{k}g_j^{n^{j}}.\]
Then there exists a nontrivial $\ell$-th horizontal character $\eta_{\ell}$ of height at most $\le (M/\delta)^{\exp(O_{\mc{P}}(D^{O_k(1)}))}$ such that 
\[\snorm{\eta_\ell(g_\ell)}_{\mb{R}/\mb{Z}} = 0.\]
\end{theorem}
\begin{proof}
Let 
\[g(n) = \exp\bigg(\sum_{\ell=1}^{k}\wt{g}_in^{i}\bigg),\]
with $\wt{g}_i \in \mf{g}_{\ell} := \log(G_\ell)$; that is, $\mf{g}_\ell$ is the Lie algebra associated with $G_{\ell}$.

We will frequently need to operate within the Lie algebra $\mf{g}^t := \log(G)^t$. We identify this vector space with $\mf{g}\otimes \mb{R}^t$ via the following map on elementary tensors (in the backwards direction)
\[a\otimes v\mapsto (v_1a, \ldots,  v_ta).\] Note that under this identification we inherit the Lie bracket operation
\[[a\otimes v_1, b\otimes v_2] = [a,b] \otimes v_1v_2,\]
where the final product is entrywise. 

Recall that 
\[G^{\Psi} = \sang{g_\ell^{v_\ell}:g_\ell\in G_\ell, v_\ell\in \Psi^{[\ell]}}.\] To see that $(g(x+P_1(y)),\ldots,g(x+P_t(y)))$ takes values in $G^\Psi$, we just need to observe that $((x+P_1(y))^\ell, \ldots, (x+P_t(y))^\ell)$ lies in $\Psi^{[\ell]}$ for each $\ell$. But this is clear since 
\[\{((x+P_1(y))^\ell, \ldots, (x+P_t(y))^\ell)\}_{x,y\in \mb{Z}} \subset \{((x+P_1^*(\vect{y}))^\ell, \ldots, (x+P_t^{\ast}(\vect{y}))^\ell)\}_{x\in \mb{Z},\vect{y}\in \mb{Z}^D}.\] 

We now note that the complexity of $G^{\Psi}/\Gamma^{\Psi}$ can easily be bounded by $M^{O_{\mc{P}}(D^{O_k(1)})}$. To see this, take an adapted basis $\{\mc{X}_1,\ldots,\mc{X}_D\}$ of $\log(G)$ of the appropriate complexity where $\{\mc{X}_{D-\on{dim}(G_\ell)+1},\ldots,\mc{X}_{D}\}$ spans $\log(G_{\ell})$ for each $\ell$. We take $w_{\ell,1},\ldots,w_{\ell,\on{dim}(\Psi^{[\ell]})-\on{dim}(\Psi^{[\ell-1]})}$ to be a sequence of vectors with integral coordinates which together with $\Psi^{[\ell-1]}$ span $\Psi^{[\ell]}$. We then have that 
\[\cup_{1\le \ell\le k}\cup_{\substack{1\le i\le \on{dim}(\Psi^{[\ell]})-\on{dim}(\Psi^{[\ell-1]})\\1\le j\le \on{dim}(G_{\ell})}}\mc{X}_{D-j+1}\otimes w_{\ell,i}\]
gives an adapted weak basis for $G^{\Psi}$. The desired complexity bound then follows via \cite[Lemma~B.11]{Len23b}.

We are ready to apply \cref{thm:total-equi} to \cref{eq:noneq}. It follows that there exists a nontrivial horizontal character $\eta$ on $G^{\Psi}/\Gamma^\Psi$ of height at most $(M/\delta)^{\exp(D^{O_{k,\mc{P}}(1)})}$ such that 
\[\snorm{\eta((g(x+P_1(y)),\ldots,g(x+P_t(y))))}_{C^{\infty}[N,N]} = 0.\]
This implies that $\eta((g(x+P_1(y)),\ldots,g(x+P_t(y))))$ may be written as an integral linear combination of $\binom{x}{i}\binom{y}{j}$. Scaling $\eta$ by a nonzero constant factor depending on $k$ and on $\mc{P}$, we may in fact assume that $\eta((g(x+P_1(y)),\ldots,g(x+P_t(y))))\in \mb{Z}[x,y]$. 

Recalling the expression for $g$ in the Lie algebra and letting $\wt \eta$ be the lift of $\eta$ to the Lie algebra,\footnote{Recall that the exponential map on a Lie subgroup is given just by restriction, and so we may write explicitly $\wt \eta(a_1,\ldots, a_t):= \eta(\exp(a_1),\ldots, \exp(a_t))$ on $\log(G^\Psi)$.} we obtain   
\begin{align*}
\eta(g(x+P_1(y)),\ldots,g(x+P_t(y)))&=\wt \eta\left(\sum_{\ell=1}^k  \wt{g}_{\ell}(x+P_1(y))^\ell,\ldots, \sum_{\ell=1}^k \wt{g}_{\ell}(x+P_t(y))^\ell \right) \\ 
&= \wt \eta\left(\sum_{\ell=1}^k \wt{g}_{\ell} \otimes ((x+P_1(y))^\ell,\ldots,(x+P_t(y))^\ell)\right).
\end{align*}
Let $v_{(i,j),\ell}\in \mb{Z}^t$ be the vector of coefficients of $x^iy^j$ in the expansion of $((x+P_1(y))^\ell, \ldots, (x+P_t(y))^\ell)$ with respect to the monomial basis. Then continuing from the computation above, and recalling that $\wt \eta$ is a Lie algebra homomorphism and so in particular is linear, we have 
\begin{align*}
\eta(g(x+P_1(y)),\ldots,g(x+P_t(y))) &= \wt \eta\left(\sum_{\ell=1}^k\sum_{i,j\geq 1} \wt{g}_{\ell} \otimes (v_{(i,j),\ell}x^iy^j) \right)\\
&= \sum_{i,j\geq 1} \sum_{\ell=1}^k \wt \eta(\wt{g}_{\ell} \otimes v_{(i,j),\ell})x^iy^j,
\end{align*}
by the bilinearity of $\otimes$ and the linearity of $\wt \eta$. Thus we deduce that 
\begin{equation}\label{eq:suminZ}
\sum_{\ell=1}^k\wt \eta(\wt{g}_{\ell} \otimes v_{(i,j),\ell})\in \mb{Z},
\end{equation} 
for all $i,j$.

Next we will claim that the transferability of $\mc{P}$ implies firstly that these $v_{(i,j),\ell}$ span exactly $\Psi^{[\ell]}$, and secondly that there exists an integer $C_{\mc{P},k}\ge 1$ such that for each $i,j,\ell$ we have 
\begin{equation}\label{eq:CPZ}
\wt \eta(\wt{g}_{\ell} \otimes v_{(i,j),\ell}) \in C_{\mc{P},k}^{-1} \cdot \mb{Z}.
\end{equation}
Let $V\leq (\mb{R}^t)^k$ be the $\mb{R}$-span over $(i,j)$ of $(v_{(i,j),1},\ldots, v_{(i,j),k})$. Note that $V$ (or rather, $V^\perp$) precisely encodes the set of algebraic relations admitted by $\mc{P}$ (up to power $k$). In fact after a moment's thought, we see that transferability implies\footnote{In fact, is equivalent to, if $k$ is sufficiently large depending on $\mc{P}$.} that 
\begin{equation}\label{eq:transfer}
V = \Psi^{[1]}\times \Psi^{[2]}\times \cdots \times \Psi^{[k]}.
\end{equation}
Towards the second part of the claim, note that we may view each $\wt \eta(\wt g_{\ell}\otimes\cdot)$ as a dual vector of $\mb{R}^t$, and so we may view $w:=(\wt \eta(\wt g_1\otimes\cdot),\ldots, \wt \eta(\wt g_t\otimes\cdot))$ as dual to $(\mb{R}^t)^k$. From \eqref{eq:suminZ} we see that in fact $w \in V_{\mb{Q}}^\ast$ (that is, $w$ maps the rational points on $V$ to $\mb{Q}$). But \eqref{eq:transfer} says that each $(0,\ldots, v_{(i,j),\ell},\ldots, 0)$ (where the nonzero entry is in the $\ell$th place) is a rational point in $V$. Therefore $w(0,\ldots, v_{(i,j),\ell},\ldots, 0)\in \mb{Q}$, that is $\wt \eta(\wt{g}_{\ell}\otimes v_{(i,j),\ell})\in \mb{Q}$. Moreover, it is clear that the denominators introduced when writing $(0,\ldots, v_{(i,j),\ell},\ldots, 0)$ as a $\mb{Q}$-linear combination of $(v_{(i,j),1},\ldots, v_{(i,j),k})$ depend only on $\mc{P}$ and $k$. This proves \eqref{eq:CPZ}.

Finally, we wish to show the existence of $i,j,\ell$ such that $\wt \eta(\exp(\cdot) \otimes v_{(i,j),\ell})$ is a nontrivial $\ell$-th horizontal character (potentially up to a scalar multiple depending on $\mc{P}, k$). This will complete the proof by noting that $\wt{g}_{\ell} = \log(g_\ell) \mod \langle G_\ell, [G_j, G_{\ell -j}] \rangle$, as per \cref{l:polyseq-liealg}. To this end, let $\ell_0$ be maximal such that some $\wt \eta(\cdot \otimes v_{(i,j),\ell_0})$ is nontrivial on $\mf{g}_{\ell_0}$; such an $\ell_0$ must exist because otherwise $\wt \eta$ would be trivial on $\log(G^\Psi)$, and hence $\eta$ would be trivial on $G^\Psi$, a contradiction. 

Now, the flag property implies that $v_{(i,j),{\ell_0}}\in \Psi^{[\ell_0+1]}$. Also, we showed above that the set of $v_{(i,j),\ell_0+1}$ span $\Psi^{[\ell_0+1]}$, and so using again the linearity of $\wt \eta$ and the bilinearity of $\otimes$, we may write the map $\wt \eta(\cdot \otimes v_{(i,j),\ell_0})$ as a linear combination of maps each of which are trivial on $\mf{g}_{\ell_0+1}$ by maximality. Thus $\wt \eta(\mf{g}_{\ell_0+1} \otimes v_{(i,j),\ell_0})=0$, and so $\wt \eta(\exp(\cdot) \otimes v_{(i,j),\ell_0})$ vanishes on $G_{\ell_0+1}$. 

Next, since any element of $[\mf{g}_\ell,\mf{g}_{\ell_0-\ell}]$ may be written as a linear combination of elements $[b_\ell, b_{\ell_0-\ell}]$ where $b_i\in \mf{g}_i$, to show that  $\wt \eta([\mf{g}_{\ell},\mf{g}_{\ell_0-\ell}] \otimes v_{(i,j),\ell_0})=0$ it suffices to show that  $\wt \eta([b_\ell,b_{\ell_0-\ell}] \otimes v_{(i,j),\ell_0})=0$. Note that from definitions we in fact have that $v_{(i,j),\ell_0}=v_{(i,j),\ell}v_{(i,j),\ell-\ell_0}$, and therefore 
\[[b_\ell,b_{\ell_0-\ell}] \otimes v_{(i,j),\ell_0}=[b_\ell\otimes v_\ell ,b_{\ell_0-\ell}\otimes v_{\ell_0-\ell}]\in [\log(G^\Psi),\log(G^\Psi)].\]
Since $\wt \eta$ is a Lie algebra homomorphism from $\log(G^\Psi)$ to $\mb{R}$, it must vanish on commutators. This proves that $\wt \eta([\mf{g}_\ell,\mf{g}_{\ell_0-\ell}] \otimes v_{(i,j),\ell_0})=0$, so $\wt \eta(\exp(\cdot)\otimes v_{(i,j),\ell_0})$ is trivial on $[G_\ell,G_{\ell_0-\ell}]$. 

Finally, the above facts and the Baker--Campbell--Hausdorff formula yields that $\wt \eta(\log(\cdot)\otimes v_{(i,j),\ell_0})$ is indeed a character from $G_{\ell_0}$ to $\mb{R}$. Also, $\wt \eta(\log(\Gamma_\ell) \otimes v_{(i,j),\ell})\in \mb{Z}$ since $\log(\Gamma_\ell) \otimes v_{(i,j),\ell}\in \log(\Gamma^\Psi)$.  This completes the proof that $\wt \eta(\log(\cdot)\otimes v_{(i,j),\ell})$ is an $\ell_0$-th horizontal character. Thus, recalling \eqref{eq:CPZ}, after scaling by $C_{\mc{P},k}$ we obtain an $\ell_0$-th horizontal character $\eta_{\ell_0} := C_{\mc{P},k}\wt \eta(\log(\cdot)\otimes v_{(i,j),\ell_0})$ such that 
\[ \snorm{\eta_{\ell_0}(g_{\ell_0})}_{\mb{R}/\mb{Z}} = 0.\]
The claim about its complexity follows from noting that the Mal'cev basis for $G/\Gamma$ is a rational linear combination of that of $G^\Psi/\Gamma^\Psi$ of with coefficients of height at most $M^{O_{\mc{P}}(D^{O_{\mc{P}}(1)}))}$, which we discussed earlier when bounding the complexity of $G^{\Psi}/\Gamma^{\Psi}$.
\end{proof}

\subsection{Proof of \cref{thm:transference}}
We first recall the statement of the inverse theorem over $\mb{Z}/N\mb{Z}$. 

\begin{theorem}\label{thm:main-inv-ff}
Fix $\delta\in (0,1/2)$. Suppose that $f\colon\mb{Z}/N\mb{Z}\to\mb{C}$ is $1$-bounded and
\[\snorm{f}_{U^{s+1}[N]}\ge\delta.\]
Then there exists a nilmanifold $G/\Gamma$ of degree $s$, complexity at most $M$, and dimension at most $D$, an $N$-periodic polynomial sequence $g(n)$ on $G$, as well as a function $F$ on $G/\Gamma$ which is at most $K$-Lipschitz such that 
\[|\mb{E}_{n\in \mb{Z}/N\mb{Z}}[f(n)\ol{F(g(n)\Gamma)}]|\ge\eps,\]
where we may take
\[D\le\log(1/\delta)^{O_s(1)}\emph{ and }\eps^{-1},K,M\le\exp(\log(1/\delta)^{O_s(1)}).\]    
\end{theorem}
\begin{proof}
The only difference in this statement versus that of \cref{thm:main} is that we insist that the underlying polynomial sequence is $N$-periodic. This is accomplished via an argument of Manners \cite{Man14}. The precise quantitative details are exactly as in \cite[Theorem~C.3]{LSS23} modulo changing the subscripts $U^4$ to $U^{s+1}$ (the primary technical work of quantifying the construction of Manners is carried out in the appropriately generality in \cite[Proposition~C.2]{LSS23}).
\end{proof}

We are ready to prove \cref{thm:transference}. Aspects of the strategy that are similar to what has been seen before are treated with some brevity. 
\begin{proof}[{Proof of \cref{thm:transference}}]
For $i=1,\ldots, t$, define the \textit{dual functions}
\[ \mc{D}_i(x) := \mb{E}_{y,\vect{y}} \prod_{j\ne i}f_i(x+P_j(y)-P_i(y))-\prod_{j\ne i}f_i(x+P^\ast_j(\vect{y})-P^\ast_i(\vect{y})).\]
 Let 
 \[ \delta := \left|\Lambda_{\mc{P}}(f_1,\ldots,f_t) - \Lambda_{\mc{P}^\ast}(f_1,\ldots,f_t)\right|.\] After making a change of variables $x\mapsto x+P_1(y)$ in $\Lambda_{\mc{P}}$ and $x \mapsto x+P_1^\ast(\vect{y})$ in $\Lambda_{\mc{P}^\ast}$, we have
\[\delta = \left|  \mb{E}_{x\in \mb{Z}/N\mb{Z}}f_1(x)\mc{D}_1(x) \right|.\]
By Cauchy-Schwarz, we have $\delta^2 \leq |\mb{E}_{x\in \mb{Z}/N\mb{Z}}\mc{D}_1(x)^2|$, and so, after substituting back in for one of the $\mc{D}_1(x)$s and reversing the change of variables above, we obtain
\[\delta^2 \leq \left| \Lambda_{\mc{P}}(\mc{D}_1,\ldots, f_t) - \Lambda_{\mc{P}^\ast}(\mc{D}_1,\ldots, f_t) \right|.\]

Now, after using the triangle inequality, we invoke \cref{thm:ff-control} to deduce then that there exists $k_{\mc{P}}$ and $\beta_\mc{P} \in (0,1]$ such that
\[ \delta^2 \leq ||\mc{D}_1||_{U^{k_\mc{P}}(\mb{Z}/N\mb{Z})}^{-\beta_{\mc{P}}}+||\mc{D}_1||_{U^{k_{\mc{P}^*}}(\mb{Z}/N\mb{Z})} + O_{\mc{P}}(N^{-\beta_{\mc{P}}}).\]
Henceforth for brevity we will let $k=\max (k_{\mc P},k_{\mc P}^\ast)$. By the monotonicity of Gowers norms, the previous equation is true if we replace $k_{\mc P}$ and $k_{\mc P^\ast}$ with $k$. Also note that we may of course assume that $O_{\mc{P}}(N^{-\beta_{\mc{P}}})\leq \delta^2/2$ (or else we have immediately proven the theorem), and so $||\mc{D}_1||_{U^k(\mb{Z}/N\mb{Z})}\gg \delta^{O_\mc{P}(1)}$.  Therefore, by the inverse theorem for Gowers norms \cref{thm:main-inv-ff}, there is an $N$-periodic nilsequence $\psi_1:= F_1(g_{1}(\cdot)\Gamma_1)$, of degree $k-1$, of complexity $\wt M_1\leq \exp(O(\log(1/\delta)^{O_k(1)}))$, dimension $D_1 \leq O(\log(1/\delta)^{O_k(1)})$, with $F_1$ being $K_1\leq \exp(O(\log(1/\delta)^{O_k(1)}))$-Lipschitz such that $|\mb{E}_x\psi_1(x)\mc{D}_1(x)|\geq \exp(-O(\log(1/\delta)^{O_k(1)}))$, that is
\[\exp(-O(\log(1/\delta)^{O_k(1)}))\leq \left| \Lambda_{\mc{P}}(\psi_1,f_2,\ldots, f_t) - \Lambda_{\mc{P}^\ast}(\psi_1,f_2,\ldots, f_t)\right|.\]
Since $k$ depends only on $\mc{P}$, we will henceforth denote constants implicit in big-O notation which depend on $k$ with a subscript of $\mc{P}$.

Finally, we will briefly justify reducing to the case that $g_1(0)=\on{id}_{G_1}$. Write  $g_1(0) = \{g_1(0)\}[g_1(0)]$ where $\{g_1(0)\}$ lies in a fundamental domain for $G_1/\Gamma_1$ (which is chosen that the Mal'cev coordinates of its elements are bounded, \cite[Lemma A.14]{GT12}) and $[g_1(0)]\in \Gamma_1$. Let $g_1'(n) = \{g_1(0)\}^{-1} g_1(n) [g_1(0)]^{-1}$, so $g_1'(0)=\on{id}_{G_1}$. Then $\psi_1(n) = F_1'(g_1(n)'\Gamma_1)$, where $F_1'(g\Gamma_1):= F_1(\{g_1(0)\}g\Gamma)$. Furthermore, the boundedness of Mal'cev coordinates for the fundamental domain of $G_1/\Gamma_1$ together with the approximate left invariance of the metric on $G_1/\Gamma_1$ (see \cite[Lemma A.5]{GT12}, and \cite[Lemma B.4]{Len23} for quantitative details) yields that $F_1'$ is $\exp(O_{\mc{P}}(\log(1/\delta)^{O_{\mc{P}}(1)}))$-Lipschitz. Thus we may and will assume without loss of generality that $g_1(0)=\on{id}_{G_1}$.

To this point we have managed to replace the function $f_1$ with a nilsequence $\psi_1$ (with poly-log dimension and quasi-polynomial complexity), at the expense of replacing $\delta$ with $\exp(-O(\log(1/\delta)^{O_{\mc{P}}(1)}))$. Rerunning the argument at each index one by one, and noting that the shape of these bounds does not change under this iteration, we may conclude that 
\begin{equation}\label{eq:nilseq-gap}
\exp(-O_{\mc{P}}(\log(1/\delta)^{O_{\mc{P}}(1)}))\leq \left| \Lambda_{\mc{P}}(\psi_1,\psi_2,\ldots, \psi_t) - \Lambda_{\mc{P}^\ast}(\psi_1,\psi_2,\ldots, \psi_t)\right|,
\end{equation}
where each $\psi_i$ may be described exactly as $\psi_1$ was above (where constants in the big-O notation may depend on $\mc{P}$). 

To apply the distributional theory we developed above, it will be convenient to view these nilsequences as living on the same nilmanifold. In particular, if $\psi_i$ is a nilsequence on $G_i/\Gamma_i$, we may view the pointwise product $\prod_{i=1}^t\psi_i$ as a nilsequence on the product nilmanifold $\prod_{i=1}^t G_i/ \prod_{i=1}^t\Gamma_i=:G/\Gamma$ with natural filtration, where the polynomial sequence is now the product $g:=(g_1,\cdots, g_t)$. Let $D$ be the dimension of $G/\Gamma$,\footnote{Except if the dimension of $G/\Gamma$ is 1, then set $D=2$ for convenience with inequalities involving expressions in $D$.} and let $M'$ be its complexity. By using \cite[Fact~3.9]{LSS24b}, these complexity parameters for $G/\Gamma$ as well as the Lipschitz constant for each $F_i$ are of the same shape as their antecedents on each $G_i/\Gamma_i$; that is, $D= O(\log(1/\delta)^{O_{\mc{P}}(1)})$, $M' = \exp(O_{\mc{P}}(\log(1/\delta)^{O_{\mc{P}}(1)}))$, and $F_i$ are $\exp(O_{\mc{P}}(\log(1/\delta)^{O_{\mc{P}}(1)}))$-Lipschitz on $G/\Gamma$ (for convenience, let $M'$ be chosen of this form so that it also upper bounds the Lipschitz norms of $F_i$). For convenience later, we will insist that $M'$ is also larger than the left hand side of \eqref{eq:nilseq-gap}. We are ready to apply \cref{thm:factor}.

Let $A\geq 2$ be a parameter to be decided upon shortly and assume (for a contradiction later) that 
\begin{equation}\label{eq:pbound}
N \geq M'^{A^{(D+2)^{O_{\mc{P}}(1)}}},
\end{equation} so that we are able to apply \cref{thm:factor} to $g$ on $G/\Gamma$ to obtain $g=g'\gamma$. Let $G'$ be the subgroup containing $g'$ which is produced by \cref{thm:factor} and let $M\leq M'^{A^{(D+2)^{O_{\mc{P}}(1)}}}$ be the parameter produced by that theorem. In particular,
\begin{equation}\label{eq:irrational-coeffs}
\snorm{\eta(g'_\ell)}_{\mb{R}/\mb{Z}} \geq M^A\cdot N^{-\ell},
\end{equation} 
for every $\ell$-th horizontal character of height at most $M^A$ on $G'/\Gamma'$. Now, $G'/\Gamma'$ has complexity $M^{O_{\mc{P}}(D^{O_{\mc{P}}(1)})}$, and furthermore each $F_i$ is $M^{O_{\mc{P}}(D^{O_{\mc{P}}(1)})}$-Lipschitz on $G'/\Gamma'$ by \cite[Lemma A.17]{GT12} (and \cite[Lemma B.9]{Len23} for quantitative details). Finally, let $F= \prod_{i=1}^t F_i$ on $G'^t$; one may check by invoking the same lemmas from \cite{GT12} and \cite{Len23} that $F$ is $\ll_{\mc{P}}M^{O_{\mc{P}}(D^{O_{\mc{P}}(1)})}$-Lipschitz with respect to a Mal'cev basis on the Leibman group $G'^\Psi/\Gamma'^\Psi$.

We are ready to continue from \eqref{eq:nilseq-gap}. By the triangle inequality we have at least one of
\begin{equation}\label{eq:equid-fails-poly}
\exp(-O_{\mc{P}}(\log(1/\delta)^{O_{\mc{P}}(1)}))\leq \left| \Lambda_{\mc{P}}(\psi_1,\psi_2,\ldots, \psi_t) - \int_{G'^\Psi/\Gamma'^\Psi}F\right|,
\end{equation}
or 
\begin{equation}\label{eq:equid-fails-linear}
\exp(-O_{\mc{P}}(\log(1/\delta)^{O_{\mc{P}}(1)}))\leq \left| \Lambda_{\mc{P}^\ast}(\psi_1,\psi_2,\ldots, \psi_t) - \int_{G'^\Psi/\Gamma'^\Psi}F\right|.
\end{equation}
Let us assume that \eqref{eq:equid-fails-poly} occurs (if not, \eqref{eq:equid-fails-linear} occurs, and one may run a directly analogous argument, where a directly analogous proof of \cref{thm:transfer} goes through.) Applying  \cref{thm:transfer}\footnote{And recalling that we insisted above that $M'$ is larger than the left hand side of \eqref{eq:nilseq-gap} so that $\delta^{-1}$ in the statement of \cref{thm:transfer} may be taken to be of size $\ll_{\mc{P}}M^{O_{\mc{P}}(D^{O_{\mc{P}}(1)})}$.} allows us to deduce that there is some $\ell$-th horizontal character $\eta_\ell$ of height $\ll_{\mc{P}}M^{\exp(O_{\mc{P}}(D^{O_{\mc{P}}(1)}))}$ such that 
\[\snorm{\eta_\ell(g'_\ell)}_{\mb{R}/\mb{Z}}=0.\]
We may therefore set $A$ to be of size $\ll \exp({O_{\mc{P}}(D^{O_{\mc{P}}(1)})})$ and obtain a contradiction with \eqref{eq:irrational-coeffs}. Thus we have contradicted \eqref{eq:pbound}, so we must have 
\[ N < M'^{\exp({O_{\mc{P}}(D^{O_{\mc{P}}(1)})})} \ll \exp(\exp(O_{\mc{P}}(\log(1/\delta)^{O_{\mc{P}}(1)}))).\]
That is, 
\[\delta < \exp(-c_{\mc{P}}(\log \log N)^{c_{\mc P}})),\]
completing the proof.
\end{proof}
\appendix
\section{Conventions regarding nilsequences}\label{sec:nilmanifolds}

We begin this appendix by giving the precise definition of the complexity of a nilmanifold; this definition is exactly as in \cite[Definition~6.1]{TT21}.
\begin{definition}\label{def:nilmanifold}
Let $s\ge 1$ be an integer and let $K>0$. A \emph{filtered nilmanifold $G/\Gamma$ of degree $s$ and complexity at most $K$} consists of the following:
\begin{itemize}
    \item a nilpotent, connected, and simply connected Lie group $G$ of dimension $m$, which can be identified with its Lie algebra $\log G$ via the exponential map $\exp\colon\log G\to G$;
    \item a filtration $G_{\bullet} = (G_i)_{i\ge 0}$ of closed connected subgroups $G_i$ of $G$ with 
    \[G = G_0 = G_1\geqslant G_1\geqslant \cdots\geqslant G_s\geqslant G_{s+1} = \mr{Id}_G\]
    such that $[G_i,G_j]\subseteq G_{i+j}$ for all $i,j\ge 0$ (equivalently $[\log G_i, \log G_j]\subseteq \log G_{i+j}$);
    \item a discrete cocompact subgroup $\Gamma$ of $G$; and
    \item a linear basis $\mathcal{X}=\{X_1,\ldots, X_{m}\}$ of $\log G$, known as a \emph{Mal'cev basis}.
\end{itemize}
We, furthermore, require that this data obeys the following conditions:
\begin{enumerate}
    \item for $1\le i,j\le m$, one has Lie algebra relations
    \[[X_i,X_j] = \sum_{i,j<k\le m}c_{ijk}X_k\]
    for rational numbers $c_{ijk}$ of height at most $K$ (we will often refer to these as the \emph{Lie bracket structure constants});
    \item for each $1\le i\le s$, the Lie algebra $\log G_i$ is spanned by $\{X_j\colon m-\dim(G_i)<j\le m\}$; and
    \item the subgroup $\Gamma$ consists of all elements of the form $\exp(t_1X_1)\cdots\exp(t_{m}X_{m})$ with $t_i\in \mb{Z}$.
\end{enumerate}
\end{definition}
We note that the conditions imply $[G,G_s]=\mr{Id}_G$, i.e., $G_s$ is contained in the center of $G$ (commutes with every element).

Next, we will define polynomial sequences in filtered nilpotent groups. This concrete definition is equivalent (by~\cite[Lemma 6.7]{GT12}) to the one given in~\cite{GT12}.
\begin{definition}\label{def:polyseq}
We adopt the conventions of \cref{def:nilmanifold}. Let $G$ be a filtered nilpotent group of degree $s$. A function $g\colon\mb{Z}\to G$ is a \emph{polynomial sequence} if there exist elements $g_i\in G_{i}$ for $i=0,\ldots,s$ such that
\begin{equation*}
g(n)=g_0g_1^{\binom{n}{1}}\cdots g_s^{\binom{n}{s}},
\end{equation*}
where $\binom{n}{i}=\frac{1}{i!}\prod_{j=0}^{i-1}(n-j)$, for all $n\in\mb{Z}$. We say a polynomial sequence $g(n)$ is $N$-periodic with respect to a lattice $\Gamma$ if \[g(n)g(n+N)^{-1}\in \Gamma\] for all $n\in \mb{Z}$.
\end{definition}
We will denote the set of polynomial sequences $g\colon\mb{Z}\to G$ relative to the filtration $G_\bullet$ of $G$ by $\on{poly}(\mb{Z},G_\bullet)$. It turns out that $\on{poly}(\mb{Z},G_\bullet)$ is a group under the natural multiplication of sequences--this is due to Lazard~\cite{Lazard54} and Leibman~\cite{Leibman98,Leibman02}.

Now we can define Mal'cev coordinates, the explicit metrics on $G$ and $G/\Gamma$ used in our work, and the precise definition of the Lipschitz norm of functions on $G/\Gamma$. These definitions are exactly as in \cite[Appendix~A]{GT12}.
\begin{definition}\label{def:Lip}
We adopt the conventions of \cref{def:nilmanifold}. Given a Mal'cev basis $\mc{X}$ and $g\in G$, there exists $(t_1,\ldots,t_m)\in \mb{R}^{m}$ such that 
\[g = \exp(t_1X_1 + t_2X_2+\ldots t_mX_m).\]
We define \emph{Mal'cev coordinates of first kind $\psi_{\exp}=\psi_{\exp,\mathcal{X}}\colon G\to\mb{R}^m$ for $g$ relative to $\mc{X}$} by 
\[\psi_{\exp}(g) := (t_1,\ldots,t_m).\]
Given $g\in G$ there also exists $(u_1,\ldots,u_m)\in \mb{R}^m$ such that 
\[g = \exp(u_1X_1)\cdots\exp(u_mX_m),\]
and we define the \emph{Mal'cev coordinates of second kind $\psi=\psi_{\mathcal{X}}\colon G\to\mb{R}^m$ for $g$ relative to $\mc{X}$} by
\[\psi(g) := (u_1,\ldots, u_m).\]
We then define a metric $d = d_{\mc{X}}$ on $G$ by
\[d(x,y) := \inf\bigg\{\sum_{i=1}^{n}\min(\snorm{\psi(x_ix_{i+1}^{-1})},\snorm{\psi(x_{i+1}x_{i}^{-1})})\colon n\in\mb{N}, x_1,\ldots,x_{n+1}\in G, x_1 = x, x_{n+1} = y\bigg\},\]
where $\snorm{\cdot}$ denotes the $\ell^\infty$-norm on $\mb{R}^m$, and define a metric on $G/\Gamma$ by
\[d(x\Gamma,y\Gamma) = \inf_{\gamma,\gamma'\in\Gamma}d(x\gamma,y\gamma').\]
Furthermore, for any function $F\colon G/\Gamma\to\mb{C}$, we define 
\[\snorm{F}_{\mr{Lip}} := \snorm{F}_{\infty} + \sup_{\substack{x,y\in G/\Gamma \\ x\neq y}}\frac{|F(x)-F(y)|}{d(x,y)}.\]
\end{definition}

We recall the notion of a horizontal character and the notion of a function $F$ having a vertical frequency; our definitions are exactly as in \cite[Definitions~1.5,~3.3,~3.4,~3.5]{GT12}.
\begin{definition}\label{def:characters}
Given a filtered nilmanifold $G/\Gamma$, the \emph{horizontal torus} is defined to be \[(G/\Gamma)_{\mr{ab}}:=G/[G,G]\Gamma.\] A \emph{horizontal character} is a continuous homomorphism $\eta\colon G\to\mb{R}/\mb{Z}$ that annihilates $\Gamma$; such characters may be equivalently viewed as characters on the horizontal torus. A horizontal character is \emph{nontrivial} if it is not identically zero. 

Furthermore, if the nilmanifold $G/\Gamma$ has degree $s$, the vertical torus is defined to be \[G_s/(G_s\cap \Gamma).\] A \emph{vertical character} is a continuous homomorphism $\xi\colon G_s\to\mb{R}/\mb{Z}$ that annihilates $\Gamma\cap G_s$. Setting $m_s = \dim G_s$, one may use the last $m_s$ coordinates of the Mal'cev coordinate map to identify $G_s$ and $G_s/(G_s\cap \Gamma)$ with $\mb{R}^{m_s}$ and $\mb{R}^{m_s}/\mb{Z}^{m_s}$, respectively. Thus, we may identify any vertical character $\xi$ with a unique $k\in \mb{Z}^{m_s}$ such that $\xi(x) = k\cdot x$ under this identification $G_s/(\Gamma\cap G_s) \cong \mb{R}^{m_s}/\mb{Z}^{m_s}$. We refer to $k$ as the \emph{frequency} of the character $\xi$, we write $|\xi| :=\snorm{k}_{\infty}$ to denote the magnitude of the frequency $\xi$, and say that a function $F\colon G/\Gamma\to\mb{C}$ \emph{has a vertical frequency $\xi$} if 
\[F(g_s\cdot x) = e(\xi(g_s))F(x)\]
for all $g_s\in G_s$ and $x\in G/\Gamma$. We may also refer to $|\xi|$ as the \textit{height} of $\xi$.
\end{definition}

We will also need the following notion of smoothness of polynomials which take values in the torus.
\begin{definition}\label{def:smoothness}
    Let $p:\mb{Z}^\ell \to \mb{R}/\mb{Z}$ be a polynomial of the form
    \[ p(\vec n) = \sum_{\vec j} \alpha_{\vec j} \binom{\vec n}{\vec j},\]
    where $\binom{\vec n}{\vec j} = \prod_{i=1}^\ell \binom{n_i}{j_i}$.
    Then, for $N_1, \ldots, N_\ell \in \mb{N}$, define: 
    \[ \snorm{p}_{C^\infty[\vec N]} := \sup_{\vec j \ne \vec 0} \snorm{\alpha_{\vec j}}_{\mb{R}/\mb{Z}}\vec{N}^{\vec j},\]
    where $\vec{N}^{\vec{j}} = \prod_{i=1}^\ell N_i^{j_i}$.
\end{definition}

We finally record the following technical lemma about coefficients of polynomial sequences when lifted to the Lie algebra.

\begin{lemma}\label{l:polyseq-liealg}
Let $G$ have a degree $k$ filtration $G_{\bullet}$ and $g(n) = g_1^{n}\cdots g_k^{n^k}$ be a polynomial sequence adapted to $G_\bullet$. Let $\mf{g}$ be the Lie algebra of $G$ and $\mf{g}_\bullet$ be the corresponding filtration. Then $\log(g(n)):= \sum_{\ell=1}^k \wt{g}_\ell n^\ell$ is a polynomial sequence adapted to $\mf{g}_\bullet$ and \[\exp(a_\ell) = g_\ell \mod \langle G_{\ell+1}, [G_j,G_{\ell-j}]\rangle,\]
for all $\ell$.
\end{lemma}
\begin{proof}
That $\log(g(n))$ is adapted to $\mf{g}_\bullet$ is an easy induction using the Baker--Campbell--Hausdorff formula. For each $\ell$, let $\mf{h}_\ell$ be the smallest Lie subalgebra of $\mf{g}$ containing $\mf{g}_{\ell+1}$ and $[\mf{g}_j,\mf{g}_{\ell-j}]$. The following consequence of the Baker--Campbell--Hausdorff formula was observed in \cite[Lemma 5.8]{Alt22b}: for univariate polynomial sequences  $q,q'$ in $\mf{g}$ adapted to $\mf{g}_\bullet$, we have 
\[(\log(\exp(q)\exp(q')))_\ell = (q)_\ell + (q')_\ell \mod \mf{h}_\ell,\]
where $(q)_\ell$ denotes the coefficient of the degree $\ell$ term of $q$. Iterating this we see that see that for each $\ell$
\[\wt{g}_\ell = (\log(\prod_{i=1}^k g_i^{n^i}))_\ell = \sum_{i=1}^k (\log(g_i^{n^i}))_\ell \mod \mf{h}_\ell = \log(g_\ell) \mod \mf{h}_\ell.\]
Exponentiating yields the result.
\end{proof}

\section{Gowers and Gowers-Peluse norm control}
We now recall a series of estimates regarding the Gowers and Gowers--Peluse norms.  
\begin{lemma}\label{lem:box-norm-conv}
Fix $s \ge 2$, $C\ge 1$, $\delta\in (0,1/2)$ and let $N \ge 1$ be an integer. Let $f:\mb{Z}\to \mb{C}$ be a $1$-bounded function supported on $[\pm N]$ and let $\mu$ be the uniform measure on $[\pm N]$. 
\begin{enumerate}
    \item[\textup{(i)}] If $\snorm{f}_{U^s[N]}\ge \delta$, then $\snorm{f}_{U_{\GP}^{s}[N;\mu]}\ge \delta^{O_s(1)}$.
    \item[\textup{(ii)}] If $\snorm{f}_{U^s_{\GP}[N;\mu]}\ge \delta$, then $\snorm{f}_{U^s[N]}\ge \delta^{O_s(1)}$.
    \item[\textup{(iii)}] Let $\mu_i,\mu_i'$ be probability measures on $\mb{Z}$ with $\mu_i(\cdot)\le C \cdot \mu_i'(\cdot)$. Then 
    \[\snorm{f}_{U_{\GP}[N;\mu_1',\ldots,\mu_s']}\ge C^{-s/2^{s-1}} \snorm{f}_{U_{\GP}[N;\mu_1,\ldots,\mu_s]}.\]
    \item[\textup{(iv)}] Let $\mu'$ be either the uniform measure on $[\pm N/C]$ or $q \cdot [\pm N/q]$ with $1\le q\le C$. If $\snorm{f}_{U^s_{\GP}[N;\mu]}\ge \delta$, then $\snorm{f}_{U_{\GP}^{s}[N;\mu']}\ge \delta^{O_s(1)}$. 
\end{enumerate}
\end{lemma}
\begin{proof}
We have that (ii) is established as \cite[Lemma~A.3]{KKL24} and (iii) and (iv) are established in \cite[Lemma~3.5]{KKL24b}, we only prove (i). Let $\tilde\mu$ denote the uniform measure on $[\pm 3N]$. We proceed inductively, proving by downwards induction for $j = s, s-1,\dots, 0$ that 
\begin{equation}\label{inductive-hyp-gowers-gp} \sum_{x,h_1,\ldots,h_j\in \mb{Z}}\mb{E}_{\substack{h_i,h_i'\sim \tilde\mu \\j+1\le i\le s}}\Delta_{h_1}\cdots \Delta_{h_j}\Delta_{(h_{j+1},h_{j+1}')}\cdots\Delta_{(h_s,h_s')}f(x)\ge c^{j - s} \delta^{2^s}  N^{j+1}\end{equation}
for a constant $c = c_s>0$. Note that the assumption in (i) corresponds to the case $j = s$, and that the case $j = 0$ gives $\snorm{f}_{U^s_{\GP}[N; \tilde\mu]}\ge c^{-s/2^s} \delta$, which immediately implies (i) by using (iii).

Assume we know \cref{inductive-hyp-gowers-gp} for some value of $j$; we shall deduce it for $j-1$. Replacing the dummy variable $h_j$ by $t$ and permuting the derivatives (which commute), the statement \cref{inductive-hyp-gowers-gp} may be written as  
\[\sum_{t\in \mb{Z}}\sum_{x,h_1,\ldots,h_{j-1}\in \mb{Z}}\mb{E}_{\substack{h_i,h_i'\sim \tilde\mu \\j+1\le i\le s}}\Delta_{h_1}\cdots \Delta_{h_{j-1}}\Delta_{(h_{j+1},h_{j+1}')}\cdots\Delta_{(h_s,h_s')}(\Delta_{t}f(x))\ge c^{j-s} \delta^{2^s} N^{j+1}.\]

As $s\ge 2$ one may observe that for fixed $t$, the internal sum is always non-negative. Furthermore the sum is only nonzero when $t\in [\pm 2N]$. As the sum in $x$ ranges over all integers, for any $\ell \in \mb{Z}$ we may substitute $x + \ell$ for $x$, and thereby replace $\Delta_{t} f$ by $\Delta_{(\ell,\ell + t)}f$, leaving the sum unchanged. Averaging over such $\ell$ weighted by $\mu$ gives
\[\sum_{t \in [\pm 2N]}\sum_{\ell \in \mb{Z}} \mu(\ell) \sum_{x,h_1,\ldots,h_{j-1}\in \mb{Z}}\mb{E}_{\substack{h_i,h_i'\sim \tilde\mu\\j+1\le i\le s}}\Delta_{h_1}\cdots \Delta_{h_{j-1}}\Delta_{(h_{j+1},h_{j+1}')}\cdots\Delta_{(h_s,h_s')}(\Delta_{(\ell, \ell + t)}f(x))\ge c^{j-s} \delta^{2^s}N^{j+1}, \] which on substituting $h_j$ for $\ell$ and $h'_j$ for $\ell + t$ gives 
\[\sum_{\substack{h_j, h'_j \in \mb{Z} \\ h'_j - h_j \in [\pm 2N]}} \mu(h_j) \sum_{x,h_1,\ldots,h_{j-1}\in \mb{Z}}\mb{E}_{\substack{h_i,h_i'\sim \tilde\mu\\j+1\le i\le s}}\Delta_{h_1}\cdots \Delta_{h_{j-1}}\Delta_{(h_{j+1},h_{j+1}')}\cdots\Delta_{(h_s,h_s')}(\Delta_{(h_j, h'_j)}f(x))\ge c^{j-s} \delta^{2^s} N^{j+1}. \] 
The sum on the left is over a subset of all pairs $(h_j, h'_j)$ with $h_j, h'_j \in [\pm 3N]$, with the weight attached to each pair being $(2N+1)^{-1} (4N + 1)^{-1} \le 5 (6N + 1)^{-2}$. It follows (using the nonnegativity of the inner sum over $x, h_1,\dots, h_{j-1}$) that
\begin{align*}
\mb{E}_{h_j,h_j'\sim \tilde\mu}&\sum_{x,h_1,\ldots,h_{j-1}\in \mb{Z}}\mb{E}_{\substack{h_i,h_i'\sim \tilde\mu\\j+1\le i\le s}}\Delta_{h_1}\cdots \Delta_{h_{j-1}}\Delta_{(h_{j+1},h_{j+1}')}\cdots\Delta_{(h_s,h_s')}(\Delta_{(h_j,h_j')}f(x))\ge  5^{-1}c^{j- s}\delta^{2^s}N^{j}\ge c^{j- s-1}\delta^{2^s}N^{j}.
\end{align*}
This is \cref{inductive-hyp-gowers-gp} in the case $j-1$ and thus completes the proof of the inductive step, and hence of item (i) in the lemma.
\end{proof}

We next require the following standard Fourier approximation lemma. 
\begin{lemma}\label{lem:major-arc-Fourier2}
Let $\delta\in (0,1/2)$ and $N\ge \delta^{-1}$. There exists $G_{\delta}:\mb{Z}\to [0,1]$ such that 
\begin{align*}
\int_{\Theta} |\wh{G_{\delta}}(\Theta)|&\ll \delta\log(1/\delta)\\
G_{\delta}(x) &= 1\text{ for } x\in [\delta N, (1-\delta) N]\\
G_{\delta}(x) &= 0\text{ for } x\notin [-\delta N, (1+\delta) N].
\end{align*}
\end{lemma}
\begin{proof}
Let 
\[G_{\delta}(x) = \frac{\mbm{1}[\cdot\in [N]] \ast \mbm{1}[\cdot\in [\delta \cdot N]]}{2\lfloor \delta N\rfloor + 1}.\]
As 
\[\wh{\mbm{1}[\cdot\in [X]]}(\Theta)\ll \min(X,\snorm{\Theta}_{\mb{R}/\mb{Z}}^{-1})\]
the result follows via direct integration. 
\end{proof}

We next recall a special case of \cite[Theorem~6.1]{Pel20} which will serve as the primary PET control statement over the integers for our work. 

\begin{theorem}\label{thm:int-control}
Let $N\ge 1$, $C\ge 1$, $1\le d'\le d$ and $a_1,\ldots,a_t\in\mb{Z}$ be distinct integers. Let $\ell$ be a nonzero integer. 

Let $f_i:[\pm N]\to \mb{C}$, $P(y) = \sum_{j=d'}^{d}b_iy^{j}\in \mb{Z}[y]$ with $b_{d'}\neq 0$ and define $P_W(y) = W^{-d'}P(Wy)$. Suppose that $\max_{i}|b_i|\le C$, $\max_{i}|a_i|\le C$, that $M(|\ell| \cdot W^{-d+d'}/N)^{1/d}\in [1/C,C]$ and  that 
\[\sum_{x\in [N]}\sum_{y\in [M]}f_1(x+a_1 \cdot \ell P_W(y))\cdots f_t(x+a_t\cdot \ell P_W(y))\ge \delta NM.\]

Let $\wt{\mu}$ denote the uniform measure on $\ell W^{d-d'} \cdot [\pm N/(\ell \cdot W^{d-d'})]$. Then there exists $s = s(d,t)\ge 2$ such that 
\[\min_{i\in [t]}\snorm{f_i}_{U_{\on{GP}^{s}[N;\wt{\mu}]}}\gg_{C} \delta^{O_{d,t}(1)}.\]
\end{theorem}
\begin{proof}
There are several slight technical differences between \cref{thm:int-control} and the statement appearing in \cite[Theorem~6.1]{Pel20}; we perform the necessary maneuvers to reduce to the statement there. 

If $a_i\neq 0$ for all $i$, we encode the $x\in [N]$ via introducing a function $f_{t+1}(x) = \mbm{1}[x\in [N]]$ and adding $a_{t+1} = 0$. If there exits $a_i = 0$, we replace $f_i$ by $\wt{f_i}(x) = f_i(x) \cdot \mbm{1}[x\in [N]]$. Observe that if $\snorm{\wt{f_i}}_{U^s_{\on{GP}}[N;\wt{\mu}]}\ge \delta'$ then taking $\delta'' = \delta'^{O_s(1)}$ we have that $\snorm{f_i \cdot G_{\delta''}}_{U^s_{\on{GP}}[N;\wt{\mu}]}\ge \delta'/2$ using \cref{lem:major-arc-Fourier2}. Via applying Fourier inversion $G_{\delta}(x) = \int_{0}^1 \wh{G_{\delta}}(\Theta)e(\Theta x)~d\Theta$ and using that the Gowers--Peluse norm satisfies triangle inequality, we have $\sup_{\Theta}\snorm{f_i\cdot e(\Theta \cdot )}_{U^s_{\on{GP}}[N;\wt{\mu}]}\ge (\delta')^{O_s(1)}$. As $s\ge 2$, a direct computation gives $\snorm{f_i\cdot e(\Theta \cdot )}_{U^s_{\on{GP}}[N;\wt{\mu}]} = \snorm{f_i}_{U^s_{\on{GP}}[N;\wt{\mu}]}$. Therefore it suffices to operate under the alternate assumption that 
\[\sum_{x\in \mb{Z}}\sum_{y\in [M]}f_1(x+a_1 \cdot \ell P_W(y))\cdots f_t(x+a_t\cdot \ell P_W(y))\ge \delta NM.\]

Observe that in this situation, by changing variables $x\to x+ 2\ell \cdot \max_{i}|a_i| P_W(y)$ we may additionally assume that $a_i>0$. Furthermore by shifting $f_i$ we may instead assume that $\on{supp}(f_i)\subseteq [C^2 N, 2C^2N]$. Given these shifted constraints, we may additionally add $f_0(x) = \mbm{1}[x\in [4C^2N]]$.

We then apply \cite[Theorem~6.1]{Pel20} with $N' = 4C^2N$ and $M$ as defined and obtain that 
\[\min_{i\in [t]}\snorm{f_i}_{U_{\on{GP}}[N;\mu_1,\ldots,\mu_s]}\gg_{C} \delta^{O_{d,t}(1)}.\]
Here $\mu_i$ are uniform measures on sets of the form $\ell W^{d-d'} \cdot b_d \cdot (a_i - a_j) d! \cdot [M^{d}]$ or $\ell W^{d-d'} \cdot b_d \cdot a_i d! \cdot [M^{d}]$. Applying \cref{lem:box-norm-conv}(iii-iv) then immediately gives the desired result. 
\end{proof}

We now state the version of PET control which is required for \cref{s:ff}; the statement we require appears as \cite[Proposition~2.2]{Pel19}.

\begin{theorem}\label{thm:ff-control}
Let $P_1(y),\ldots,P_m(y)\in \mb{Z}[y]$ be pairwise distinct polynomials and $f_1,\ldots,f_m:\mb{Z}/N\mb{Z}\to \mb{C}$ be $1$-bounded. Then there exists $s = s(P_1,\ldots,P_m)\ge 1$, $\beta = \beta(P_1,\ldots,P_m)\in (0,1]$ such that 
\[\Big|\mb{E}\Big[\prod_{i=1}^{m}f_i(x+P_i(y))\Big]\Big|\le \min_{i}\snorm{f_i}_{U^s(\mb{Z}/N\mb{Z})}^{\beta}+ O_{\mc{P}}(N^{-\beta}).\]
\end{theorem}

\bibliographystyle{amsplain0-full.bst}
\bibliography{main.bib}

\providecommand{\bysame}{\leavevmode\hbox to3em{\hrulefill}\thinspace}
\providecommand{\MR}{\relax\ifhmode\unskip\space\fi MR }
\providecommand{\MRhref}[2]{%
  \href{http://www.ams.org/mathscinet-getitem?mr=#1}{#2}
}
\providecommand{\href}[2]{#2}
\begin{thebibliography}{10}

\bibitem{Alt22b}
Daniel Altman, \emph{A non-flag arithmetic regularity lemma and counting
  lemma}, arXiv:2209.14083.

\bibitem{BL96}
V.~Bergelson and A.~Leibman, \emph{Polynomial extensions of van der {W}aerden's
  and {S}zemer\'{e}di's theorems}, Journal of the American Mathematical Society
  \textbf{9} (1996), 725--753.

\bibitem{BC17}
J.~Bourgain and M.-C. Chang, \emph{Nonlinear {R}oth type theorems in finite
  fields}, Israel Journal of Mathematics \textbf{221} (2017), 853--867.

\bibitem{CS14}
Pablo Candela and Olof Sisask, \emph{Convergence results for systems of linear
  forms on cyclic groups and periodic nilsequences}, SIAM Journal on Discrete
  Mathematics \textbf{28} (2014), 786--810.

\bibitem{DLS20}
Dong Dong, Xiaochun Li, and Will Sawin, \emph{Improved estimates for polynomial
  {R}oth type theorems in finite fields}, Journal d'Analyse Math\'{e}matique
  \textbf{141} (2020), 689--705.

\bibitem{Gow98}
W.~T. Gowers, \emph{A new proof of {S}zemer\'edi's theorem for arithmetic
  progressions of length four}, Geometric and Functional Analysis \textbf{8}
  (1998), 529--551.

\bibitem{Gow01}
W.~T. Gowers, \emph{Arithmetic progressions in sparse sets}, Current
  developments in mathematics, 2000, Int. Press, Somerville, MA, 2001,
  pp.~149--196.

\bibitem{Gow01a}
W.~T. Gowers, \emph{A new proof of {S}zemer\'{e}di's theorem}, Geometric and
  Functional Analysis \textbf{11} (2001), 465--588.

\bibitem{Gre02}
Ben Green, \emph{On arithmetic structures in dense sets of integers}, Duke
  Mathematical Journal \textbf{114} (2002), 215--238.

\bibitem{Gre05}
Ben Green, \emph{Roth's theorem in the primes}, Annals of Mathematics. Second
  Series \textbf{161} (2005), 1609--1636.

\bibitem{GS24}
Ben Green and Mehtaab Sawhney, \emph{Primes of the form $ p^2+ nq^2$},
  arXiv:2410.04189.

\bibitem{GT09}
Ben Green and Terence Tao, \emph{New bounds for {S}zemer\'{e}di's theorem.
  {II}. {A} new bound for {$r_4(N)$}}, Analytic number theory, Cambridge Univ.
  Press, Cambridge, 2009, pp.~180--204.

\bibitem{GT10}
Ben Green and Terence Tao, \emph{Linear equations in primes}, Annals of
  Mathematics. Second Series \textbf{171} (2010), 1753--1850.

\bibitem{GT12}
Ben Green and Terence Tao, \emph{The quantitative behaviour of polynomial
  orbits on nilmanifolds}, Annals of Mathematics. Second Series \textbf{175}
  (2012), 465--540.

\bibitem{IK04}
Henryk Iwaniec and Emmanuel Kowalski, \emph{Analytic number theory}, American
  Mathematical Society Colloquium Publications, vol.~53, American Mathematical
  Society, Providence, RI, 2004.

\bibitem{KM23}
Zander Kelley and Raghu Meka, \emph{Strong bounds for 3-progressions},
  arXiv:2302.05537.

\bibitem{KKL24}
Noah Kravitz, Borys Kuca, and James Leng, \emph{Corners with polynomial side
  length}, arXiv:2407.08637.

\bibitem{KKL24b}
Noah Kravitz, Borys Kuca, and James Leng, \emph{Quantitative concatenation for
  polynomial box norms}, arXiv:2407.08636.

\bibitem{Kuc21}
Borys Kuca, \emph{Further bounds in the polynomial {S}zemer\'edi theorem over
  finite fields}, Acta Arithmetica \textbf{198} (2021), 77--108.

\bibitem{Kuc23}
Borys Kuca, \emph{On several notions of complexity of polynomial progressions},
  Ergodic Theory and Dynamical Systems \textbf{43} (2023), 1269--1323.

\bibitem{Lazard54}
M.~Lazard, \emph{Sur les groupes nilpotents et les anneaux de {L}ie}, Annales
  Scientifiques de l'\'{E}cole Normale Sup\'{e}rieure. Troisi\`eme S\'{e}rie
  \textbf{71} (1954), 101--190.

\bibitem{Leibman98}
A.~Leibman, \emph{Polynomial sequences in groups}, Journal of Algebra
  \textbf{201} (1998), 189--206.

\bibitem{Leibman02}
A.~Leibman, \emph{Polynomial mappings of groups}, Israel Journal of Mathematics
  \textbf{129} (2002), 29--60.

\bibitem{Len23b}
James Leng, \emph{Efficient {E}quidistribution of {N}ilsequences},
  arXiv:2312.10772.

\bibitem{Len23}
James Leng, \emph{The partition rank vs. analytic rank problem for cyclic
  groups {I}. {E}quidistribution for periodic nilsequences}, arXiv:2306.13820.

\bibitem{Len24}
James Leng, \emph{A quantitative bound for {S}zemer\'edi's theorem for a
  complexity one polynomial progression over {$\Bbb Z/{N}\Bbb Z$}}, Discrete
  Analysis (2024), Paper No. 3, 33.

\bibitem{LSS24c}
James Leng, Ashwin Sah, and Mehtaab Sawhney, \emph{Improved {B}ounds for
  {S}zemer\'{e}di's {T}heorem}, arXiv:2402.17995.

\bibitem{LSS24b}
James Leng, Ashwin Sah, and Mehtaab Sawhney, \emph{Quasipolynomial bounds for
  the inverse theorem for the {G}owers {$U^{s+1}[N]$}-norm}, arXiv:2402.17994.

\bibitem{LSS23}
James Leng, Ashwin Sah, and Mehtaab Sawhney, \emph{Improved bounds for
  five-term arithmetic progressions}, Mathematical Proceedings of the Cambridge
  Philosophical Society \textbf{177} (2024), 371--413.

\bibitem{Luc06}
J.~Lucier, \emph{Intersective sets given by a polynomial}, Acta Arithmetica
  \textbf{123} (2006), 57--95.

\bibitem{Man14}
F.~Manners, \emph{Periodic nilsequences and inverse theorems on cyclic groups},
  arXiv:1404.7742.

\bibitem{Man21}
Freddie Manners, \emph{True complexity and iterated {C}auchy--{S}chwarz},
  arXiv:2109.05731.

\bibitem{Man18}
Frederick Manners, \emph{{Q}uantitative bounds in the inverse theorem for the
  {G}owers ${U}^{s+1}$-norms over cyclic groups}, arXiv:1811.00718.

\bibitem{Pel18}
Sarah Peluse, \emph{Three-term polynomial progressions in subsets of finite
  fields}, Israel Journal of Mathematics \textbf{228} (2018), 379--405.

\bibitem{Pel19}
Sarah Peluse, \emph{On the polynomial {S}zemer\'{e}di theorem in finite
  fields}, Duke Mathematical Journal \textbf{168} (2019), 749--774.

\bibitem{Pel20}
Sarah Peluse, \emph{Bounds for sets with no polynomial progressions}, Forum of
  Mathematics. Pi \textbf{8} (2020), e16, 55.

\bibitem{PP22}
Sarah Peluse and Sean Prendiville, \emph{A polylogarithmic bound in the
  nonlinear {R}oth theorem}, International Mathematics Research Notices. IMRN
  (2022), 5658--5684.

\bibitem{PP24}
Sarah Peluse and Sean Prendiville, \emph{Quantitative bounds in the nonlinear
  {R}oth theorem}, Inventiones Mathematicae \textbf{238} (2024), 865--903.

\bibitem{PSS23}
Sarah Peluse, Ashwin Sah, and Mehtaab Sawhney, \emph{Effective bounds for
  roth's theorem with shifted square common difference}, arXiv:2309.08359.

\bibitem{Pre17}
Sean Prendiville, \emph{Quantitative bounds in the polynomial {S}zemer\'{e}di
  theorem: the homogeneous case}, Discrete Analysis (2017), Paper No. 5, 34.

\bibitem{Sar78}
Andr\'{a}s S\'ark\"{o}zy, \emph{On difference sets of sequences of integers.
  {II}}, Annales Universitatis Scientiarum Budapestinensis de Rolando
  E\"otv\"os Nominatae. Sectio Mathematica \textbf{21} (1978), 45--53.

\bibitem{Sze70}
E.~Szemer\'{e}di, \emph{On sets of integers containing no four elements in
  arithmetic progression}, Number {T}heory ({C}olloq., {J}\'{a}nos {B}olyai
  {M}ath. {S}oc., {D}ebrecen, 1968), Colloq. Math. Soc. J\'{a}nos Bolyai,
  vol.~2, North-Holland, Amsterdam-London, 1970, pp.~197--204.

\bibitem{Sze75}
E.~Szemer\'{e}di, \emph{On sets of integers containing no {$k$} elements in
  arithmetic progression}, Polska Akademia Nauk. Instytut Matematyczny. Acta
  Arithmetica \textbf{27} (1975), 199--245.

\bibitem{Tao12}
Terence Tao, \emph{Higher order {F}ourier analysis}, Graduate Studies in
  Mathematics, vol. 142, American Mathematical Society, Providence, RI, 2012.

\bibitem{TT21}
Terence Tao and Joni Ter\"av\"ainen, \emph{Quantitative bounds for {G}owers
  uniformity of the {M}\"obius and von {M}angoldt functions}, Journal of the
  European Mathematical Society (JEMS) \textbf{27} (2025), 1321--1384.

\bibitem{WZ12}
Trevor~D. Wooley and Tamar~D. Ziegler, \emph{Multiple recurrence and
  convergence along the primes}, American Journal of Mathematics \textbf{134}
  (2012), 1705--1732.

\end{thebibliography}

\end{document}